\documentclass{article}

\usepackage{amsmath,amssymb,amsthm,comment}
\usepackage{tikz}
\usetikzlibrary{decorations.pathreplacing,intersections,calc,positioning}

\usepackage[backref,hyperindex,colorlinks,citecolor=blue]{hyperref}

\newcommand{\N}{\mathbb{N}}
\newcommand{\Z}{\mathbb{Z}}
\newcommand{\R}{\mathbb{R}}

\newcommand{\forc}{\mathcal{V}}

\newcommand{\uniform}[1]{\hat{#1}}
\newcommand{\prob}{\mathrm{Pr}}
\newcommand{\onleft}[3]{L^{#2}_{#3}}
\newcommand{\onright}[3]{R^{#2}_{#3}}
\newcommand{\tuple}[1]{\langle #1 \rangle}

\newcommand{\meas}[1]{\mathcal{M}(#1)}

\newtheorem{theorem}{Theorem}
\newtheorem{proposition}[theorem]{Proposition}
\newtheorem{lemma}[theorem]{Lemma}
\newtheorem{definition}[theorem]{Definition}
\newtheorem{corollary}[theorem]{Corollary}

\newtheorem{observation}[theorem]{Observation}
\newtheorem{example}[theorem]{Example}

\newcommand{\customqed}[1]{{\renewcommand{\qedsymbol}{#1}\qed}}
\newcommand{\varqed}{\customqed{\hbox{$\lrcorner$}}} 

\theoremstyle{definition} 
\newtheorem{Premark}{\color{blue}Peter remark}

\title{Stable Multi-Level Monotonic Eroders}

\author{P\'eter G\'acs\textsuperscript{1} and Ilkka T\"orm\"a\textsuperscript{2}\footnote{Second author supported by a Fulbright Finland scholarship and Academy of Finland grant 295095} \\
        \textsuperscript{1}Department of Computer Science, Boston University \\
        \textsuperscript{2}Department of Mathematics and Statistics, University of Turku, \\ \texttt{iatorm@utu.fi}}

\date{\today}

\begin{document}

\maketitle

\begin{abstract}
Eroders are monotonic cellular automata with a linearly ordered state set that eventually wipe out any finite island of nonzero states.
One-dimensional eroders were studied by Gal'perin in the 1970s, who presented a simple combinatorial characterization of the class.
The multi-dimensional case has been studied by Toom and others, but no such characterization has been found.
We prove a similar characterization for those one-dimensional monotonic cellular automata that are eroders even in the presence of random noise.
\end{abstract}

\section{Introduction}

Cellular automata, CA for short, are a class of dynamical systems that are discrete in both space and time.
A CA consists of a finite or infinite grid of identical finite state machines, usually called cells, that interact with only finitely many neighbors.
At each time step, each machine synchronously enters a new state based on the current states of itself and its neighbors.
The grid is usually assumed to be homogeneous, that is, each machine has a neighborhood of identical shape and uses the same update function.

In this article, we consider \emph{monotonic} cellular automata with \emph{random errors}.
In a monotonic CA, the state set of the finite state machines is totally ordered, and the dynamics of the system respects this order.
They are closely related to \emph{bootstrap percolation}, introduced in \cite{ChLeRe79}, in which the state set is $\{0, 1\}$, a cell cannot change its state from $1$ to $0$, and the system is initiated from a random configuration.
In general, we denote by $0$ the bottom element of the state set.
The dynamical properties of arbitrary monotonic CA in dimension $1$ were studied by Gal'perin \cite{Ga76,Ga77}, who gave a computable characterization of those automata that erase arbitrary finite islands in a sea of $0$-states.
We introduce randomness into the model by allowing each machine to make an error with some fixed small probability, turning it into a probabilistic cellular automaton.
See \cite{MaMa14} and references therein for a survey on this topic.
In fact, since we allow the errors of different states to be dependent, our model is somewhat more general than probabilistic cellular automata.
We prove a version of Gal'perin's result in this extended setting.
More explicitly, we characterize those one-dimensional monotonic cellular automata for which the asymptotic density of $0$-states in the trajectory started from the all-$0$ configuration tends to zero with the error rate.
We call a CA that satisfies this condition a \emph{stable eroder}.
As a corollary, we also show that it is decidable whether a given monotonic automaton is a stable eroder.

The proof of our result is split into two sections, one for each direction of the equivalence.
One direction of the proof, that our condition only holds for stable eroders, uses combinatorial objects that record the causal relationships between nonzero states in a trajectory of the CA, inspired by the work of Toom \cite{To80}.
We show that a nonzero state is always accompanied by such an object, as well as a set of errors whose size is comparable to the size of the object.
A counting argument then bounds the probability of the nonzero state.
For the converse direction, we prove that if a CA that does not satisfy our condition, one can find a finite island of nonzero states that persists forever with an arbitrarily high probability.
We show how randomly occurring errors will ``repair'' the borders and internal structure of the island faster than the CA can erode it away.

\section{Definitions}

Let $\Z^d$ denote the $d$-dimensional integer lattice.
We fix a finite state set $S$, endow it with the discrete topology, and endow $S^{\Z^d}$ with the product topology.
The \emph{shift by $\vec{n} \in \Z^d$} is the function $\sigma^{\vec{n}} : S^\Z \to S^\Z$ defined by $\sigma^{\vec{n}}(x)_{\vec{v}} = x_{\vec{v} + \vec{n}}$ for all $x \in S^{\Z^d}$ and $\vec{v} \in \Z^d$.
If $d = 1$, we denote $\sigma = \sigma^1$.
A \emph{cellular automaton} is a function from $S^{\Z^d}$ to itself that is continuous and commutes with each $\sigma^{\vec{n}}$.
By the Curtis-Lyndon-Hedlund Theorem, it has a finite neighborhood: the value of $f(x)_{\vec 0}$ depends only on finitely many coordinates of $x$.
A \emph{radius} of $f$ is an integer $r \geq 0$ such that $f(x)_{\vec 0}$ depends only on $x|_{[-r, r]^d}$.
A radius gives rise to a \emph{local rule}: a function $F : S^{[-r, r]^d} \to S$ such that $f(x)_{\vec 0} = F(x|_{[-r, r]^d})$.
Indexing an element $\eta \in S^{\Z^d \times \N}$ by $(\vec{n}, t) \in \Z^{d} \times \N$ is denoted $\eta_{\vec{n}}^t$, and $\eta^t \in S^{\Z^d}$ is the $t$'th $d$-dimensional slice of $\eta$.

Denote by $\meas{S^{\Z^d}}$ the set of Borel probability measures on $S^{\Z^d}$, and similarly for $S^{\Z^d \times \N}$.
Consider a mapping $R$ from $\meas{S^{\Z^d}}$ to $\meas{S^{\Z^d \times \N}}$.
For a measure $\mu$, we consider $R(\mu)$ as a random variable with values in $S^{\Z^d \times \N}$.
We say that $R$ is a \emph{stochastic symbolic process}, if
\begin{itemize}
\item it is continuous in the weak topologies,
\item it is linear, that is, $R(\lambda \mu + (1-\lambda) \nu) = \lambda R(\mu) + (1-\lambda)R(\nu)$ holds for $\mu, \nu \in \meas{S^{\Z^d}}$ and $0 \leq \lambda \leq 1$, and
\item $R(\mu)^0 = \mu$.
\end{itemize}
Such a mapping is determined by its images on point measures.
If $\mu$ is a point measure concentrated on $x \in S^{\Z^d}$, we denote $R(\mu) = R(x)$, and call it a \emph{random trajectory} with initial condition $x$.
We say that $R$ is an \emph{$\epsilon$-perturbation} of $f$, if for any $x \in S^{\Z^d}$ and any finite set $C \subset \Z^d \times \N$, the probability that $R(x)_{\vec{v}}^{t+1} \neq f(R(x)^t)_{\vec{v}}$ holds for all $(\vec{v} ,t) \in C$ is at most $\epsilon^{|C|}$.
In this context, the set of coordinates $(\vec{v}, t) \in \Z^d \times \N$ with $R(x)_{\vec{v}}^{t+1} \neq f(R(x)^t)_{\vec{v}}$ is usually called the \emph{error set} of $R(x)$.

For now on, consider only linearly ordered state sets: $S = \{0, 1, \ldots, m\}$ for some $m \geq 1$.
For $x, y \in S^{\Z^d}$, we denote $x \geq y$ if $x_{\vec{v}} \geq y_{\vec{v}}$ holds for all $\vec{v} \in \Z^d$.
A cellular automaton $f : S^{\Z^d} \to S^{\Z^d}$ is \emph{monotonic} if $x \geq y$ implies $f(x) \geq f(y)$ for all $x, y \in S^{\Z^d}$.
For $a \in S$, the all-$a$ configuration is denoted by $\uniform{a}$.
The state $a$ is called \emph{quiescent} for $f$, if $f(\uniform{a}) = \uniform{a}$.
We will always assume that the extremal states $0$ and $m$ are quiescent.
A configuration $x \in S^{\Z^d}$ is an \emph{$a$-island}, if $\uniform{a} \leq x$ and $x_{\vec{v}} = a$ for all but finitely many $\vec{v} \in \Z^d$.
A $0$-island will simply be called an \emph{island}.

\section{Eroders and Zero Sets}

The notion of an eroder has appeared many times in the literature with different names, including nilpotency on finite configurations.
It simply means that a CA removes all islands in a finite (but not necessarily uniform) number of steps.

\begin{definition}
  We say that a monotonic cellular automaton $f$ is an \emph{eroder}, if for every $0$-island $x \in S^{\Z^d}$ there exists $n \in \N$ such that $f^n(x) = \uniform{0}$.
\end{definition}

With the terminology of~\cite{To80}, this means that the all-$0$ trajectory is attractive.
Our aim is to extend this notion to cellular automata with random perturbations.
Of course, a nontrivial $\epsilon$-perturbation of a deterministic CA will almost surely never reach the uniform zero configuration $\uniform{0}$.
Thus we present the following definition, which is equivalent to the all-$0$ trajectory being stable, again using the terminology of~\cite{To80}.

\begin{definition}
  Let $f$ be a monotonic cellular automaton on $S^{\Z^d}$, and denote
  \[
    P(\epsilon) = \inf \{ \prob [R(\uniform{0})_{\vec{0}}^t = 0] \;|\; \mbox{$t \in \N$, $R$ is an $\epsilon$-perturbation of $f$} \}.
  \]
  We say that $f$ is a \emph{stable eroder}, if $P(\epsilon) \longrightarrow 1$ as $\epsilon \longrightarrow 0$.
\end{definition}

The following result justifies our terminology.

\begin{proposition}
  \label{prop:StableIsEroder}
  All stable eroders are eroders.
\end{proposition}

\begin{proof}
  Suppose that $f : S^{\Z^d} \to S^{\Z^d}$ is a monotonic cellular automaton that is not an eroder.
  Then there exists an island $x \in S^{\Z^d}$ such that $f^n(x) \neq \uniform{0}$ for all $n \in \N$.
  We choose a sequence of coordinates $(\vec{v}_i)_{i \in \N}$ such that $f^i(x)_{\vec{v}_i} \neq 0$.
  
  Let $\epsilon > 0$ be arbitrary, and let $R_\epsilon$ be the $\epsilon$-perturbation of $f$ where 
  \[
    R_\epsilon(y)^{t+1}_{\vec{v}} =
    \left\{
      \begin{array}{ll}
        m & \mathrm{with~probability~} \epsilon \\
        f(R_\epsilon(y)^t)_{\vec{v}} & \mathrm{with~probability~} 1 - \epsilon
      \end{array}
    \right.
  \]
  independently for every $(\vec{v}, t) \in \Z^d \times \N$ and $y \in S^{\Z^d}$, where $m \in S$ is the maximal state.
  Note that $R_\epsilon(y)^{t+1} \geq f(R_\epsilon(y)^t)$ holds for all $t \in \N$.
  Because of this, $R_\epsilon(y)^t \geq \sigma^{\vec{n}}(x)$ implies $R_\epsilon(y)^{t+i} \geq \sigma^{\vec{n}}(f^i(x))$, and thus $R_\epsilon(y)^{t+i}_{\vec{v}_i - \vec{n}} \neq 0$, for all $i \in \N$ and $\vec{n} \in \Z^d$.

  For every $\vec{n} \in \Z^d$ and $t > 0$, the probability that $R_\epsilon(\uniform{0})^t \geq \sigma^{\vec{n}}(x)$ is bounded from below by a positive constant, say $\delta > 0$.
  Then we have
  \[
    \prob[R_\epsilon(\uniform{0})_0^t \neq 0] \geq \prob[\exists i \leq t : R_\epsilon(\uniform{0})^{t-i} \geq \sigma^{\vec{v}_i}(x)] \geq 1 - (1 - \delta)^t \stackrel{t \to \infty}{\longrightarrow} 1
  \]
  Since $\epsilon$ was arbitrary, $f$ is not a stable eroder.
\end{proof}

It is known that for automata with more than two states, the converse does not hold: there exist automata that are eroders but not stable eroders.
One such example is given in \cite{To76}, and we present a slight modification of it here for completeness.

\begin{example}
  \label{ex:Ternary}
  Let $S = \{0, 1, 2\}$, and let $f : S^\Z \to S^\Z$ be the radius-1 cellular automaton defined by the local rule
  \[
    F(a,b,c) =
    \begin{cases}
      0, & \text{if~} a = 0, b \leq 1, c \leq 1, \\
      1, & \text{if~} b = 2, c \leq 1, \\
      2, & \text{if~} a + b \geq c = 2, \\
      b, & \text{otherwise.}
    \end{cases}
  \]
  A simple (yet tedious) case analysis shows that $f$ is monotonic.
  Consider the $0$-island $x = {}^\infty 0 2^N 0^\infty$ consisting of a run of $2$-states of length $N$.
  The CA $f$ erodes the $2$-states from the right: for $0 \leq t \leq N$ we have $f^t(x) = {}^\infty 0 2^{N-t} 1^t 0^\infty$.
  When only $1$-states remain, they are eroded from the left: for $0 \leq s \leq N$ we have $f^{N + s}(x) = {}^\infty 0 0^s 1^{N-s} 0^\infty$.
  In particular, $f^{2 N}(x)$ is the all-$0$ configuration.
  Since all $0$-islands are majored by an island like $x$, this implies that $f$ is an eroder.

  On the other hand, $f$ is not a stable eroder.
  We only present a high-level idea of the proof.
  Consider a perturbation $R$ of $f$ where on each space-time coordinate $(i, t) \in \Z \times \N$, an error occurs with probability $\epsilon > 0$ independently of all other coordinates, and always produces the state $2$, and consider the random trajectory $R(x)$ with $x = {}^\infty 0 2^N 0^\infty$.
  An error occurring next to the left border of the island will extend the island by one cell.
  This means that the left border performs a random walk, and moves to the left with average speed $\epsilon$, as long as it contains a $2$-state.
  Consider a coordinate $(i, t)$ on the left border of the island (where we have $i < 0$ if the border has moved to the left), and suppose that the border contains a $2$-state at all times up to $t$.
  If we have $R(x)^t_i = 1$, then $R(x)^{t-1}_{i+1} = 1$ as well, by the local rule of $f$.
  This can be extended to $R(x)^{t-s}_{i+s} = 1$ for all $s \leq N - i$.
  The probability of this event is $(1 - \epsilon)^{N-i}$, which drops exponentially as the border $i$ moves to the left.
  One can verify that as $N$ grows, the probability of maintaining the left border in state $2$ forever approaches unity.
\end{example}

\begin{definition}
  \label{def:ForcingSets}
  Let $V \subset \Z^d$ be a finite set, let $k > 0$, and let $a < b \in S$ be quiescent states.
  We say that $V$ is \emph{$a,b$-forcing at level $k$} for a CA $f$, if for every configuration $x \in S^{\Z^d}$ such that $x \leq \uniform{b}$ and $x_{\vec{v}} \leq a$ for all $\vec{v} \in V$, we have $f^k(x)_{\vec{0}} \leq a$.
  The family of all minimal $a,b$-forcing sets at level $k$ for $f$ is denoted by $\forc^k_{a,b}(f)$.
\end{definition}

Note that if $U$ is $a,b$-forcing at level $k$ and $V$ is $a,b$-forcing at level $\ell$, then $U + V$ is $a,b$-forcing at level $k + \ell$, and thus some subset $W \subset U + V$ satisfies $W \in \forc^{k + \ell}_{a,b}(f)$.
Forcing sets are analogous to zero sets as usually defined for binary cellular automata, and in that context, they can be used to characterize eroders.

\begin{definition}
  \label{def:ErodingPair}
  For a set $A \subset \R^d$, denote by $C(A)$ the convex hull of $A$.
  For a CA $f$ and quiescent states $a < b$, denote
  \[
    \tau^k_{a,b}(f) = \bigcap_{V \in \forc^k_{a,b}(f)} C(V) \subset \R^d
  \]
  where each $a,b$-forcing set $V \in \forc^k_{a,b}(f)$ is interpreted as a subset of $\R^d$.
  We say $f$ is \emph{$a,b$-shrinking}, if $\tau^k_{a,b}(f) = \emptyset$ for some $k > 0$.
\end{definition}

It is not immediately clear whether one can algorithimically decide if a cellular automaton is eroding for a given pair of states, since Definition~\ref{def:ErodingPair} refers to an unbounded variable $k$, but there may exist a bound $K$, computable from the radius of $f$ and the number of states, such that $f$ is $a,b$-shrinking precisely when $\tau^k_{a,b}(f) = \emptyset$ for some $k \leq K$.
We show by an example that it is not sufficient to consider the case $k = 1$.
However, the condition turns out to be decidable for one-dimensional cellular automata.
This follows from the results of \cite{Ga76,Ga77}, and we repeat it explicitly in Lemma~\ref{lem:ForcingRates}.
The decidability of the condition in the multi-dimensional case is left open.

\begin{example}
We show that the case $k = 1$ is not enough to determine the eroding condition for two given states.
Namely, consider the one-dimensional cellular automaton $f$ with radius $1$ and state set $S = \{0, 1, \ldots, m\}$ defined by the local rule
\[
  F(a,b,c) =
  \begin{cases}
    b-1, & \text{if~} b > 0 \text{~and~} c = 0, \\
    b, & \text{otherwise.}
  \end{cases}
\]
This automaton decrements a nonzero state by one if its right neighbor is $0$, and otherwise keeps the state fixed.
Consider the families $\forc^k_{0,m}(f)$ of $0,m$-forcing sets.
We have $\{0\} \in \forc^k_{0,m}(f)$ for all $k \geq 1$, since the state $0$ always stays as $0$.
For $1 \leq k < m$, the configuration $x = {}^\infty 0 . m 0^\infty$ satisfies $f^k(x)_0 \neq 0$, which implies $\forc^k_{0,m}(f) = \{ \{0\} \}$.
However, $f^m(y)_0 = 0$ for any $y \in S^\Z$ with $y_1 = 0$, which implies $\{1\} \in \forc^m_{0,m}(f)$.
It follows that the CA $f$ is $0,m$-shrinking, but this fact cannot be deduced from $\forc^k_{0,m}(f)$ for any $k < m$.
\end{example}

For monotonic binary cellular automata, the $0,1$-shrinking condition is always decidable, as it suffices to consider only $k = 1$.
Furthermore, it is equivalent to being an eroder and a stable eroder.

\begin{proposition}[Section IV in \cite{To80}]
  \label{prop:BinaryEroder}
  Let $S = \{0, 1\}$, and let $f$ be a monotonic CA on $S^{\Z^d}$.
  The following conditions are equivalent.
  \begin{itemize}
  \item $f$ is an eroder.
  \item $f$ is a stable eroder.
  \item $f$ is $0,1$-shrinking.
  \item $\tau^1_{a,b}(f) = \emptyset$.
  \end{itemize}
\end{proposition}

It is known that this characterization of eroders generalizes to larger alphabets in the one-dimensional case, but not in the multi-dimensional case.
This was proved in \cite{Ga76}, and we repeat it in Theorem~\ref{thm:Galperin}.
In this article, we will show that the characterization of stable eroders likewise generalizes to larger alphabets in the one-dimensional case (although the generalization is different).
The multi-dimensional case is left open.

\begin{definition}
  \label{def:StabCond}
  Let $S = \{0, 1, \ldots, m\}$, and let $f$ be a monotonic CA on $S^{\Z^d}$.
  We say that $f$ satisfies the \emph{stability condition} if there exist quiescent states $0 = a_1 < a_2 < \cdots < a_k = m$ such that $f$ is $a_i, a_{i+1}$-shrinking for all $1 \leq i < k$.
\end{definition}

  




\begin{theorem}
  \label{thm:StableEroderGeneral}
  A one-dimensional monotonic automaton $f$ is a stable eroder if and only if it satisfies the stability condition.
\end{theorem}

The next three sections are devoted to the proof of this result.

\section{Eroders in One Dimension: Gal'perin Rates}

In this section, we fix a one-dimensional monotonic automaton $f : S^\Z \to S^\Z$ with radius $r \geq 0$ and consider the evolution of certain configurations under $f$.

\begin{definition}
  An \emph{increasing ladder} is a configuration $x \in S^\Z$ such that $x_i \geq x_j$ for all $i \geq j$.
  Decreasing ladders are defined analogously.
  Let $a \neq b \in S$ be quiescent states.
  The \emph{step of type $a,b$} is the ladder $x \in S^\Z$ defined by
  \[
    x_i = \left\{
      \begin{array}{ll}
        a, & \mathrm{if~} i < 0, \\
        b, & \mathrm{if~} i \geq 0.
      \end{array} \right.
  \]
  For $t \in \N$, we denote
  \begin{align*}
    \onleft{x}{t}{a,b} & {} = \max \{ i \in \Z \;|\; f^t(x)_i = a \} \\
    \onright{x}{t}{a,b} & {} = \min \{ i \in \Z \;|\; f^t(x)_i = b \}
  \end{align*}
  In the limit, we denote $L_{a,b} = \lim_{t \to \infty} \onleft{x}{t}{a,b} / t$, and similarly for $R_{a,b}$.
\end{definition}

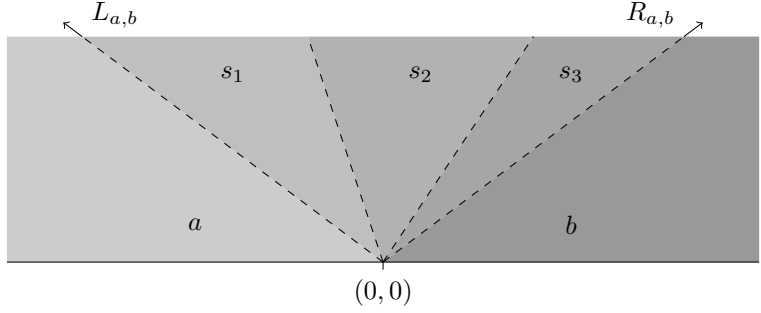
\begin{figure}[ht]
\begin{center}
\begin{tikzpicture}

  \fill [black!20] (-5,0) -- (0,0) -- (-4,3) -- (-5,3);
  \fill [black!25] (0,0) -- (-1,3) -- (-4,3);
  \fill [black!30] (0,0) -- (2,3) -- (-1,3);
  \fill [black!35] (0,0) -- (4,3) -- (2,3);
  \fill [black!40] (5,0) -- (0,0) -- (4,3) -- (5,3);

  \draw (-5,0) -- (5,0);
  \draw [dashed] (0,0) -- (-4,3);
  \draw [dashed] (0,0) -- (-1,3);
  \draw [dashed] (0,0) -- (2,3);
  \draw [dashed] (0,0) -- (4,3);
  \draw [->] (-4,3) -- ++(-1/4,3/16);
  \draw [->] (4,3) -- ++(1/4,3/16);
  
  \draw (0,-0.1) -- (0,0);
  
  \node at (-2.5,0.5) {$a$};
  \node at (2.5,0.5) {$b$};
  \node at (-2,2.5) {$s_1$};
  \node at (0.5,2.5) {$s_2$};
  \node at (2.5,2.5) {$s_3$};
  
  \node [above right] at (-4,3) {$L_{a,b}$};
  \node [above left] at (4,3) {$R_{a,b}$};
  
  \node [below] at (0,-0.1) {$(0,0)$};

\end{tikzpicture}
\caption{The Gal'perin rates of a monotonic cellular automaton. Time increases upward. The figure depicts regions of states $a < s_1 < s_2 < s_3 < b$ in space-time.}
\label{fig:Galperin}
\end{center}
\end{figure}

See Figure~\ref{fig:Galperin} for a visualization of the definition.
We know from the work of Gal'perin \cite{Ga76} that these quantities, which we call the \emph{Gal'perin rates} of $f$, always exist and are rational numbers.
Furthermore, they can be effectively computed from the local rule of $f$ \cite{Ga77,dSaLeTo14}.
The following result is also convenient.

\begin{lemma}[\cite{Ga76}]
  \label{lem:Galperin}
  There exists a constant $K > 0$ such that
  \[ \left| \onleft{}{t}{a,b} - t \cdot L_{a,b} \right| \leq K \mbox{~and~} \left| \onright{}{t}{a,b} - t \cdot R_{a,b} \right| \leq K \]
  hold for all quiescent $a \neq b \in S$ and $t \in \N$.
\end{lemma}

We now show how these rates are connected to the forcing sets of $f$.

\begin{lemma}
  \label{lem:TrivialLR}
  Let $a \neq b \in S$ be quiescent.
  Then $L_{a,b} \leq R_{a,b}$.
  If there is no quiescent state $c$ between $a$ and $b$, then $L_{a,b} = R_{a,b}$.
\end{lemma}

\begin{proof}
  The first claim is clear, since $L^t_{a,b} \leq R^t_{a,b}$ holds for all $t \in \N$.
  For the second claim, suppose that $a < b$ and no state $a < c < b$ is quiescent.
  Since $f$ is monotonic, there exists $n \geq 1$ such that $f^n(\uniform{c}) \in \{\uniform{a}, \uniform{b}\}$ for all such $c$.
  If $C = R^t_{a,b} - L^t_{a,b}$ is large enough, there exists $a < c < b$ and $L^t_{a,b} + C/3 < i < R^t_{a,b} - C/3$ such that $f^t(x)|_{[i - r n, i + r n]} = c^{2 r n + 1}$, where $r \geq 0$ is the radius of $f$.
  Then $f^{t+n}(x)_i \in \{a, b\}$, and we have either $L^{t+n}_{a,b} > L^t_{a,b} + C/3$ or $R^{t+n}_{a,b} < R^t_{a,b} - C/3$.
  Lemma~\ref{lem:Galperin} implies $|L^{t+n}_{a,b} - L^t_{a,b}| \leq n L_{a,b} + K$, and similarly for $R^t_{a,b}$, so $C$ is bounded by a constant.
  Then we have $L_{a,b} = R_{a,b}$.
\end{proof}

\begin{lemma}
  \label{lem:ForcingRates}
  Let $a < b \in S$ be quiescent.
  Then
  \begin{align*}
    L_{a,b} & {} = \sup \{ -\max U / k \;|\; k > 0, U \in \forc^k_{a,b}(f) \} \\
    R_{b,a} & {} = \inf \{ -\min V / k \;|\; k > 0, V \in \forc^k_{a,b}(f) \}
  \end{align*}
\end{lemma}

\begin{proof}
  Let $x \in S^\Z$ be the step of type $a,b$, and let $U \in \forc^k_{a,b}(f)$ be arbitrary.
  Denote $u = \max U$.
  For $t \in \N$, we have $f^t(x)_{\onleft{}{t}{a,b} + i - u} = a$ for all $i \in U$.
  Since $U$ is $a,b$-forcing, this implies $f^{t+k}(x)_{\onleft{}{t}{a,b} - u} = a$.
  Because $f$ is monotonic, $f^{t+k}(x)$ is an increasing ladder, and we have $\onleft{}{t+k}{a,b} \geq \onleft{}{t}{a,b} - u$.
  An inductive argument now shows $\onleft{}{t}{a,b} \geq - t u / k$, 
  which gives $L_{a,b} \geq - u / k$.
  
  For the other direction, let $U_t = \{- r t, - r t + 1, \ldots, - \onleft{}{t}{a,b} - 1 \} \subset \Z$ for $t > 0$, where $r \in \N$ is the radius of $f$.
  If $y \in S^\Z$ is a configuration satisfying $y \leq \uniform{b}$ and $y_i \leq a$ for all $i \in U_t$, then $y_i \leq \sigma^{\onleft{}{t}{a,b}}(x)_i$ for all $- r t \leq i \leq r t$.
  This implies $f^t(y)_0 \leq f^t(x)_{\onleft{}{t}{a,b}} = a$, so $U_t$ is an $a,b$-forcing set at level $t$.
  We also have $- \max U_t / t = (\onleft{}{t}{a,b} + 1) / t \longrightarrow L_{a,b}$ as $t$ grows.
  
  The second statement follows by symmetry.
\end{proof}

\begin{lemma}
  \label{lem:ForcingRates2}
  The following conditions are equivalent.
  \begin{enumerate}
  \item $f$ is $a,b$-shrinking.
  \item For some $k > 0$, there exist $U,V \in \forc^k_{a,b}(f)$ such that $\max U < \min V$.
  \item $L_{a,b} > R_{b,a}$.
  \end{enumerate}
  Furthermore, they are algorithmically decidable from the local rule of $f$.
\end{lemma}

\begin{proof}
  If $f$ is $a,b$-shrinking, then there exists $k \in \N$ with $\tau^k_{a,b} = \emptyset$.
  Since the convex hull $C(V)$ is an interval for each $V \in \forc^k_{a,b}$, this means that there exist forcing sets $U, V \in \forc^k_{a,b}$ with $\max U < \min V$.
  Conversely, (2) implies $U \cap V = \emptyset$, so $f$ is $a,b$-shrinking.
  Thus (1) and (2) are equivalent.
  
  If the (2) holds, then $L_{a,b} \geq -\max U / k > \min V / k \leq R_{b,a}$ by Lemma~\ref{lem:ForcingRates}, so we obtain (3).
  Suppose finally $L_{a,b} > R_{b,a}$.
  Then there exist $k, \ell \in \N$, and two sets $U \in \forc^k_{a,b}(f)$ and $V \in \forc^\ell_{a,b}(f)$ with $\max U / k < \min V / \ell$.
  Denote $U' = U + U + \cdots + U$ (a sum of $\ell$ sets), and $V' = V + V + \cdots + V$ (a sum of $k$ sets).
  By the remark after Definition~\ref{def:ForcingSets}, there exist subsets $\hat U \subset U'$ and $\hat V \subset V'$ that are in $\forc^{k \ell}_{a,b}$ and $\forc^{k \ell}_{b,a}$ respectively.
  This implies $\max \hat U \leq \ell \max U < k \min V \leq \min \hat V$.
  Thus the second condition holds.
  
  The decidability follows from the results of \cite{Ga77}.
\end{proof}

The characterization of eroders in one dimension is the following.

\begin{theorem}[\cite{Ga76}]
  \label{thm:Galperin}
  The cellular automaton $f$ is an eroder if and only if $R_{0,a} > L_{a,0}$ holds for all quiescent states $a \in S \setminus \{0\}$.
\end{theorem}

This condition is not equivalent to $f$ being $0,a$-shrinking for all quiescent $a \in S \setminus \{0\}$.
More explicitly, if $f$ is $0,m$-shrinking where $m \in S$ is the maximal state, then it is an eroder, but the converse does not hold in general.
In the one-dimensional case, see Example~\ref{ex:TernaryAgain}.
Also, the condition implies that the time required for $f$ to erode an island is  at most linear 
in the diameter of the island, and in two dimensions examples of slower eroders are known, 
even in the case that the automaton is \emph{decreasing}, that is, $f(x) \leq x$ holds 
for all $x \in S^{\Z^2}$ \cite{dMeTo06}.

We will now reformulate our main result, Theorem~\ref{thm:StableEroderGeneral}, in terms of Gal'perin rates.

\begin{theorem}
  \label{thm:StableEroder}
  Let $S = \{0, 1, \ldots, m\}$.
  A one-dimensional monotonic automaton $f$ on $S^\Z$ is a stable eroder if and only if there exist quiescent states $0 = a_1 < a_2 < \cdots < a_k = m$ such that $L_{a_i, a_{i+1}} > R_{a_{i+1}, a_i}$ for all $1 \leq i < k$.
\end{theorem}

\begin{example}
  \label{ex:TernaryAgain}
  Consider again the three-state CA $f$ of Example~\ref{ex:Ternary}.
  We compute the Gal'perin rates of $f$ for all pairs of states, since every state of $f$ is quiescent.
  Consider first $R_{0,1}$.
  In the step configuration ${}^\infty 0 . 1^\infty$, the leftmost $1$ turns into a $0$ in one application of $f$.
  Thus $R_{0,1} = 1$.
  In a similar manner, we compute
  \[
    \begin{array}{cccccc}
      L_{0,1} = 1 & L_{1,0} = 0 & L_{0,2} = 0 & L_{2,0} = -1 & L_{1,2} = -1 & L_{2,1} = -1 \\
      R_{0,1} = 1 & R_{1,0} = 0 & R_{0,2} = 0 & R_{2,0} = 0  & R_{1,2} = -1 & R_{2,1} = -1
    \end{array}
  \]
  From this table, one can check that the condition of Theorem~\ref{thm:Galperin} holds, so $f$ is an eroder.
  Note also that $f$ is not $0,2$-shrinking.
  The condition of Theorem~\ref{thm:StableEroder} does not hold, so $f$ is not a stable eroder.
\end{example}

We list here some generally useful lemmas.

\begin{lemma}
  \label{lem:MonotoneLR}
  Let $a < b \in S$ be quiescent.
  Then we have
  \begin{equation}
    \label{eq:LIneq}
    L_{a,b} \geq L_{a,b+1} \mbox{~and~} R_{b,a} \leq R_{b+1,a}
  \end{equation}
  together with
  \begin{equation}
    \label{eq:RIneq}
    R_{a,b} \geq R_{a+1,b} \mbox{~and~} L_{b,a} \leq L_{b,a+1}.
  \end{equation}
\end{lemma}

\begin{proof}
  We only prove the first inequality of \eqref{eq:LIneq}; the others follow by symmetry, either by swapping left and right, or inverting the order of the state set.
  Let $x, y \in S^\Z$ be the steps of type $a,b$ and $a,b+1$, respectively.
  Then we have $\uniform{a} \leq f^t(x) \leq f^t(y)$ for all $t \in \N$ by monotonicity, and in particular, $f^t(y)_i = a$ implies $f^t(x)_i = a$ for all $i \in \Z$.
  This implies $L^t_{a,b} \geq L^t_{a,b+1}$, and the claim follows by taking the limit.
\end{proof}

\begin{lemma}
  \label{lem:BorderGuard}
  Let $a < b \in S$ be quiescent states.
  Then there exists a quiescent state $c \in S$ with $a < c \leq b$ such that $L_{c,a} = R_{b,a}$.
\end{lemma}

Of course, there also exists $a < d \leq b$ such that $R_{a,d} = L_{a,b}$ by symmetry.


\begin{proof}
  This follows directly from Lemma 5($\delta$) of~\cite{Ga76}.
\end{proof}

\section{Stability Condition is Sufficient}

In this section, we prove the first part of Theorem~\ref{thm:StableEroder}: automata that satisfy the stability condition are stable eroders.
The proof follows the ideas presented in \cite{To80}.
The high level idea is that we take an $\epsilon$-perturbation of $f$, and consider a single coordinate $(0, T)$ in a random trajectory starting fom the uniform zero configuration.
Assuming the coordinate has a nonzero state, we construct a geometric object that records the `reason' for this event, that is, traces it back to some finite subset of coordinates where errors have occurred.
We prove that the size of this subset grows linearly with the size of the object, and there is an exponential number of objects of a given size.
A simple calculation then shows that the probability of $(0, T)$ having a nonzero state approaches 0 with $\epsilon$.
The main difference to \cite{To80} is that our geometric objects are polygons instead of trees.

For the remainder of this section, fix a monotonic cellular automaton $f$ on $S^\Z$ that satisfies the stability condition for the quiescent states $0 = a_1 < \cdots < a_k = m$.
For convenience, we denote $L_n = L_{a_n, a_{n+1}}$, $R_n = R_{a_{n+1}, a_n}$, $S_n = \{a \in S \;|\; a_n < a \leq a_{n+1}\}$ and $S_n^+ = \cup_{\ell \geq n} S_\ell$ for $n \in \{1, \ldots, k - 1\}$.
By Lemma~\ref{lem:ForcingRates}, there exist $k_n > 0$ and forcing sets $U_n, V_n \in \forc^{k_n}_{a_n, a_{n+1}}(f)$ such that $U_n, V_n \subset \{-k_n r, \ldots, k_n r\}$ and $-k_n L_n \leq u_n = \max U_n < \min V_n = v_n \leq -k_n R_n$.
Recall the number $K$ from Lemma~\ref{lem:Galperin}.
We may assume that
\begin{equation}
\label{eq:BoundOnK}
2 K < v_n - u_n, \qquad k_n = 1
\end{equation}
hold for all $n$; this can be guaranteed by taking a large enough power of $f$ that is divisible by each $k_n$.
Let $R$ be an $\epsilon$-perturbation of $f$ for some small $\epsilon > 0$, and consider the random trajectory $\eta = R(\uniform{0})$.
We will prove that $\prob[\eta_0^T > 0] \stackrel{\epsilon \to 0}{\longrightarrow} 0$ uniformly in $T$.

We define some more auxiliary concepts before proceeding with the proof.
Define four sets of integer vectors: 
\begin{align*}
  \Delta^L_n & {} = \{ (i, -1) \;|\; -r \leq i \leq u_n\}, &
  \Delta^R_n & {} = \{ (i, 1) \;|\; -r \leq i \leq -v_n \}, \\
  \Delta^C_n & {} = \{ (i, 0) \;|\; 0 < i \leq 2 r \}, &
  \Delta^B_n & {} = - \Delta^C_n
\end{align*}
Finally, define the random \emph{error set} by $E = \{ (i, t) \in \Z \times \N \;|\; \eta^{t+1}_i \neq f(\eta^t)_i \}$.

Let $T \geq 1$ be an arbitrary positive integer.
We will now construct a system of geometric shapes on the vertex set $W = \{ (i, t) \in \Z \times \N \;|\; t \leq T, |i| \leq r (T - t) \}$ parametrized by the states $S_n \subset S$ and a set $C \subset W$ such that $\eta^t_i \in S_n$ for all $(i, t) \in C$.
For a set $P \subset \R^2$, denote its border by $\partial P$.
By a \emph{polygon} we mean a closed and bounded subset of $\R^2$ that is a finite union of triangles and line segments.
In particular, a polygon may not be equal to the closure of its interior.

\begin{definition}
  \label{def:Polygon}
  For $n \in \{ 1, \ldots, k \}$, define a \emph{space-time polygon of level $n$} as a pair $\tuple{P, X}$ satisfying the following conditions.
  \begin{itemize}
  \item[a)] $P \subset \R^2$ is a non-empty simply connected polygon.
  \item[b)] $X = \tuple{w_1, \ldots, w_\ell}$ is a cyclic list of vertices in $W \cap \partial P$, which occur in counterclockwise order on $\partial P$, and $\partial P$ is the union of the line segments $\overline{w_i w_{i+1}}$.
  \item[c)] For all $i \in \{1, \ldots, \ell\}$ we have $w_{i+1} - w_i \in \Delta^L_n \cup \Delta^R_n \cup \Delta^C_n \cup \Delta^B_n$.
  \item[d)] If $(i,t), (j,t) \in X$ and $0 < |i - j| \leq 2r$, then the line segment from $(i,t)$ to $(j,t)$ is a subset of $P$.
  \item[e)] Each $w = (i,t) \in X$ satisfies $\eta^t_i \in S_n$ and one of the following conditions:
    \begin{enumerate}
    \item
      There exists $\tilde n > n$ and $j \in \{- \tilde n r, \ldots, \tilde n r\}$ with $\eta_{i+j}^{t-1} \in S_{\tilde n}$.
      Then $w$ has \emph{type 1}, the coordinate $v = (i+j, t-1) \in W$ is the \emph{support point of $w$ at level $\tilde n$}, and we denote $v = \mathrm{supp}(w)$.
      If there are several candidates for the support point, we choose one that maximizes the level $\tilde n$.
    \item
      There is an error at the coordinate that precedes $w$ in the trajectory: we have $(i, t-1) \in E$.
      Then $w$ has \emph{type 2}.
    \item
      There exist $j \in U_n$ and $j' \in V_n$ such that the triangle spanned by $w$, $(i+j, t-1)$ and $(i+j', t-1)$ is a subset of $P$.
      Then $w$ has \emph{type 3}, and $w$ is a \emph{child} of $(i+j, t-1)$ and $(i+j', t-1)$.
    \end{enumerate}
  \end{itemize}
\end{definition}

Item (d) in Definition~\ref{def:Polygon} is a technical constraint that is easy to enforce during the construction and makes the proof of Lemma~\ref{lem:BorderSize} simpler.
If the condition does not hold, we say that the pair $(i,t), (j,t)$ violates condition (d) in $\tuple{P, X}$.
Similarly, the constraint on maximizing $\tilde n$ in item (e) simplifies the proofs of Lemma~\ref{lem:VertexDisjoint} and Lemma~\ref{lem:NotInside} without affecting the construction in any other way.

See Figure~\ref{fig:Polygon} for a visualization of a space-time polygon.
The shaded area is $P$, the black dots are the vertices $w_i$ in the list $X$, and each arrow is a line segment $\overline{w_i w_{i+1}}$.
An edge is in $\Delta^C_n$ if it points east, in $\Delta^B_n$ if it points west, in $\Delta^L_n$ if southwest and in $\Delta^R_n$ if northwest.
Some edges in the figure are labeled with the set they belong to.
Note that horizontal line segments may be traversed twice, as is the case near the rightmost vertex.
The darker shaded areas at the bottom are parts of a space-time polygon of level $\tilde n$ for some $\tilde n > n$, and the dashed lines denote the support point relation.
The vertices  in $X$ with dashed lines have type 1.
Vertices enclosed in boxes are produced by errors, and they have type 2.
Other vertices have type 3; one of the triangles is depicted in the figure.

\begin{figure}[ht]
  \begin{center}
    \begin{tikzpicture}

      \coordinate (v1) at (0,0);
      \coordinate (v2) at (1.5,0);
      \coordinate (v3) at (3,0);
      \coordinate (v4) at (2.5,1);
      \coordinate (v5) at (4,1);
      \coordinate (v6) at (5.5,1);
      \coordinate (v7) at (5,0);
      \coordinate (v8) at (6.5,0);
      \coordinate (v9) at (8,0);
      \coordinate (v10) at (7.5,-1);
      \coordinate (v11) at (8.5,-1);
      \coordinate (v12) at (9.5,-1);
      \coordinate (v13) at (8,0);
      \coordinate (v14) at (7.5,1);
      \coordinate (v15) at (7,2);
      \coordinate (v16) at (6,3);
      \coordinate (v17) at (5.5,4);
      \coordinate (v18) at (5,3);
      \coordinate (v19) at (4,3);
      \coordinate (v20) at (3,4);
      \coordinate (v21) at (2,4);
      \coordinate (v22) at (1.5,3);
      \coordinate (v23) at (0.5,3);
      \coordinate (v24) at (-0.5,2);
      \coordinate (v25) at (-1,1);
      \coordinate (v26) at (0.5,1);
      \coordinate (v27) at (9,1); 

      \fill [black!20] (v1) --(v2) --(v3) --(v4) -- (v5) --(v6) --(v7) --(v8) --(v9) --(v10) --(v11) --(v12) --(v13) --(v14) --(v15) --(v16) --(v17) --(v18) --(v19) --(v20) --(v21) --(v22) --(v23) --(v24) --(v25) --(v26) -- (v1);
      
      \foreach \i in {1, ..., 27}{
        \fill (v\i) circle (0.07);
        \node (n\i) at (v\i) {};
      }
      \foreach \i [count=\ii] in {2, ..., 26}{
        \draw [thick,->] (n\ii) -- (n\i);
      }
      \draw [thick,->] (n26) -- (n1);
      \draw [thick,->]
      ($(n14.south east)!0.5!(n14.east)$) --
      ($(n27.south west)!0.5!(n27.west)$);
      \draw [thick,->]
      ($(n27.north west)!0.5!(n27.west)$) --
      ($(n14.north east)!0.5!(n14.east)$);

      \node [below] at ($(v7)!0.5!(v8)$) {$\Delta^C_n$};
      \node [above] at ($(v18)!0.5!(v19)$) {$\Delta^B_n$};
      \node [above right] at ($(v15)!0.3!(v16)$) {$\Delta^R_n$};
      \node [above left] at ($(v23)!0.7!(v24)$) {$\Delta^L_n$};

      \node (a) at (-1.5,-1) {};
      \node (a1) at (-1.5,0) {};
      \node (a2) at (0,-1) {};
      \fill [black!40] (0.5,-1.5) -- (a2.center) -- (a1.center) -- (-1.75,0) -- (-1.75,-1.5);
      \fill (a1) circle (0.07);
      \fill (a2) circle (0.07);
      \draw [thick,->] (0.5,-1.5) -- (a2);
      \draw [thick,->] (a2) -- (a1);
      \draw [thick] (a1) -- (-1.75,0);
      \draw [dashed] (v1) -- (a);
      \draw [dashed] (v25) -- (a1);

      \fill [black!40] (3.5,-1) -- (4,-1.5) -- (3,-1.5);
      \fill (3.5,-1) circle (0.07);
      \node (b) at (3.5,-1) {};
      \draw [thick] (3,-1.5) -- (b);
      \draw [thick,<-] (b) -- (4,-1.5);
      \draw [dashed] (v2) -- (b);
      \draw [dashed] (v3) -- (b) -- (v7);

      \foreach \i in {5,8,10,11,12,27}{
        \draw ($(v\i)+(-0.1,-0.1)$) rectangle ++(0.2,0.2);
      }

      \draw (v15) -- ++(-0.5,-1) -- ++(1,0);
      
      \foreach \x in {-1.5,-1,...,10}{
        \foreach \y in {-1,0,1,2,3,4}{
          \fill (\x,\y) circle (0.03cm);
        }
      }
      
    \end{tikzpicture}
  \end{center}
  \caption{A space-time polygon.}
  \label{fig:Polygon}
\end{figure}
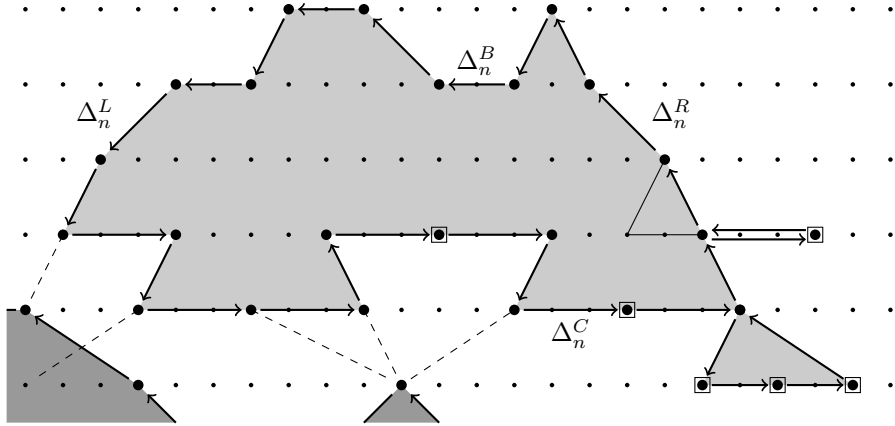

Intuitively, a space-time polygon on level $n$ records the `reason' for the vertices in the set $C$ being in the state $n$.
It resembles the notion of \emph{truss} in~\cite{To80}.
The idea of the proof is that a coordinate $(i, t) \in \Z^2$ with $\eta^t_i \neq 0$ gives rise to a finite collection of space-time polygons, a constant fraction of whose vertices have type 2, that is, are caused by errors.
The number of such collections with $N$ vertices in total grows exponentially with $N$, and by choosing the error rate $\epsilon$ small enough, we can bound the probability that any collection of polygons enables $\eta^t_i \neq 0$.
We begin by showing that in a single polygon $\tuple{P, X}$, the number of vertices of type 1 or 2 grows linearly with the size of the set $X$.

\begin{lemma}
\label{lem:BorderSize}
For all $n \in \{1, \ldots, k\}$, there exists $\delta_n > 0$ such that for any space-time polygon $\tuple{P, X}$ of level $n$, there are has at least $\delta_n |X|$ elements of type 1 or 2 in $X$.
\end{lemma}

\begin{proof}
  Let $X = \tuple{w_1, \ldots, w_\ell}$, and denote $I = \{i \in \{1, \ldots, \ell\} \;|\; w_{i+1} - w_i \in \Delta^C_n \}$ and $J = \{1, \ldots, \ell\} \setminus I$.
  Define a linear function $M_n : \R^2 \to \R$ by $M_n(i, t) = i + t (u_n + v_n)/2$.
  We claim that $M_n$ is positive on $\Delta^C_n$ and negative on $\Delta^L_n \cup \Delta^R_n \cup \Delta^B_n$.
  The case of $\Delta^C_n$ and $\Delta^B_n$ is clear, since we have $t = 0$, and $i$ is positive on $\Delta^C_n$ and negative on $\Delta^B_n$.
  Let then $(i, -1) \in \Delta^L_n$, so that $i \leq u_n$.
  Then we have $M_n(i, -1) = i - (u_n + v_n)/2 \leq (u_n - v_n)/2 < 0$, since $u_n < v_n$.
  For $(i, 1) \in \Delta^R_n$ we have $i \leq -v_n$, which implies $M_n(i, 1) = i + (u_n + v_n)/2 \leq (u_n - v_n)/2 < 0$.
  
  Since these sets are finite, there exists $0 < \delta < 1$ such that $\delta < M_n(v) < \delta^{-1}$ for all $v \in \Delta^C_n$ and $-\delta^{-1} < M_n(v) < -\delta$ for all $v \in \Delta^R_n \cup \Delta^L_n \cup \Delta^B_n$.
  Since $X$ is a circular list, we have $\sum_{i=1}^\ell M_n(w_{i+1} - w_i) = 0$ by linearity of $M_n$.
  The sets $I$ and $J$ form a partition of $\{1, \ldots, \ell\}$, so we have $\delta^2 |J| \leq |I| \leq \delta^{-2} |J|$.
  From this we also deduce $|I| \geq |X| / (1 + \delta^{-2})$.

  We present a geometric argument for the fact that the number of vertices of type 1 or 2 in $X$ is at least $|I|/2$.
  Let $i \in \{1, \ldots, \ell\}$ be such that $w_{i+1} - w_i \in \Delta^C_n$, that is, $i \in I$.
  Suppose that both $w_i$ and $w_{i+1}$ have type 3.
  Then there exist $u, v \in \Z^2$ with $w_i - u \in \Delta^R_n$, $v - w_{i+1} \in \Delta^L_n$ and $\overline{u w_i}, \overline{w_{i+1} v} \subset P$.
  The angle $\angle w_{i-1} w_i w_{i+1}$ cannot be larger than $\angle u w_i w_{i+1} < \pi$, which implies $w_i - w_{i-1} \in \Delta^R_n$.
  By a symmetric argument we have $w_{i+2} - w_{i+1} \in \Delta^L_n$.
  Denoting $w_{i-1} = (a, t)$ and $w_{i+2} = (b, s)$, a simple calculation shows that $t = s$ and $|a - b| \leq 2r$.
  We must also have $a \neq b$, since $P$ is simply connected.
  This violates condition (d), so one of $w_i$ or $w_{i+1}$ has type 1 or 2.
  See Figure~\ref{fig:BorderCount} for a visualization of this argument.
  It follows that the number of type-1 or 2 vertices in $X$ is at least $|I|/2 \geq |X| / (2 + 2\delta^{-2})$.
  We define $\delta_n = (2 + 2\delta^{-2})^{-1}$, which finishes the proof.
\end{proof}

\begin{figure}
  \begin{center}
    \begin{tikzpicture}

      \coordinate (wi) at (0,2);
      \coordinate (wii) at (6,2);
      \coordinate (u) at (0.5,0);
      \coordinate (v) at (5.5,0);
      \coordinate (wj) at (1.5,0);
      \coordinate (wk) at (4.5,0);

      \foreach \name in {wi,wii,u,v,wj,wk}{
        \node (n\name) at (\name) {};
      }
      \foreach \name in {wi,wii,wj,wk}{
        \fill (\name) circle (0.07);
      }
      \foreach \name in {u,v}{
        \draw (\name) circle (0.07);
      }
      
      \draw [thick,->] (nwj) -- (nwi);
      \draw [thick,->] (nwi) -- (nwii);
      \draw [thick,->] (nwii) -- (nwk);
      \draw [dotted] (nu) -- (nwi);
      \draw [dotted] (nwii) -- (nv);

      \node [below=0.2cm] at (u) {$u$};
      \node [below=0.2cm] at (wj) {$w_{i-1}$};
      \node [left=0.2cm] at (wi) {$w_i$};
      \node [right=0.2cm] at (wii) {$w_{i+1}$};
      \node [below=0.2cm] at (wk) {$w_{i+2}$};
      \node [below=0.2cm] at (v) {$v$};
      
    \end{tikzpicture}
  \end{center}
  \caption{A visualization of the proof of Lemma~\ref{lem:BorderSize}.}
  \label{fig:BorderCount}
\end{figure}
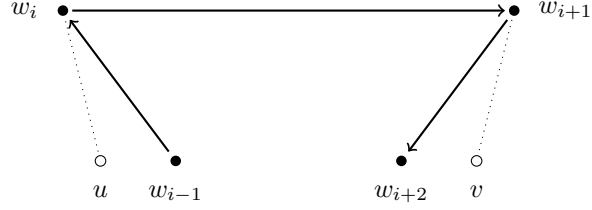

We will now construct a set $Q(C)$ of disjoint space-time polygons of level $n$ whose union contains $C$, the set of initial vertices.
We construct the polygons iteratively, maintaining a set $Q_p(C)$ of `incomplete polygons' that are allowed to violate condition (d) and the second part of condition (e) in Definition~\ref{def:Polygon}, meaning that the vertices at time $t$ may not have a type.
We start with $Q_0(C)$ being the collection of single-vertex polygons $\tuple{\{w\}, \tuple{w}}$ for $w \in C$.

Assume then that we have constructed the set $Q_p(C)$ for some $p \geq 0$.
There are three possible conditions that prevent us from choosing $Q(C) = Q_p(C)$, which we refer to as \emph{violations}:
\begin{enumerate}
\item Some vertex $w \in W$ occurs in the list of some polygon of $Q_p(C)$, and has no type in any polygon in which it occurs.
\item Some pair of vertices $v, w$ occurs on the list of some polygon of $Q_p(C)$ (not necessarily consecutively), and violates condition (d) in every polygon in which it occurs.
\item Some polygons of $Q_p(C)$ have nonempty intersection.
\end{enumerate}
If none of these violations hold, then $Q_p(C)$ consists of disjoint space-time polygons of level $n$.
Namely, if $w \in W$ is a vertex of some polygon $\tuple{P, X} \in Q_p(C)$, then it has a type in one of them (since violation 1 does not hold), which must be $\tuple{P, X}$ since the polygons are disjoint.
Similarly, if $v, w \in W$ occur in the list of some polygon $\tuple{P, X} \in Q_p(C)$, then it does not violate condition (d) in some polygon (since violation 2 does not hold), which must be $\tuple{P, X}$.
We now show how to handle the violations one by one.
The first two cases involve adding new simple polygons to the set $Q_p(C)$, and they are visualized in Figure~\ref{fig:Conditions}.
In the last case, we show how to merge two intersecting polygons into one, which is a more involved process.

\begin{figure}[ht]
  \begin{center}
    \begin{tikzpicture}


      \coordinate (v1) at (0,0);
      \coordinate (v2) at (1,0);
      \coordinate (v3) at (2,0);

      \fill [black!30] (0,0) rectangle (2,0.5);

      \foreach \i in {1,2,3}{
        \fill (v\i) circle (0.07cm);
        \node (nv\i) at (v\i) {};
      }
      \draw [thick,->] (nv1) -- (nv2);
      \draw [thick,->] (nv2) -- (nv3);

      \node at (1,-0.75) {$\Downarrow$};
      
      
      \begin{scope}[yshift=-2cm]
        \coordinate (v1) at (0,0);
        \coordinate (v2) at (1,0);
        \coordinate (v3) at (2,0);
        \coordinate (v4) at (0.5,-1);
        \coordinate (v5) at (1.5,-1);
        
        \fill [black!20] (0,0) rectangle (2,0.5);
        \fill [black!30] (v2) -- (v4) -- (v5);
        
        \foreach \i in {1,3}{
          \fill [black!50] (v\i) circle (0.07cm);
          \node (nv\i) at (v\i) {};
        }
        \foreach \i in {2,4,5}{
          \fill (v\i) circle (0.07cm);
          \node (nv\i) at (v\i) {};
        }
        \draw [thick,->,black!50] (nv1) -- (nv2);
        \draw [thick,->,black!50] (nv2) -- (nv3);
        \draw [thick,->] (nv2) -- (nv4);
        \draw [thick,->] (nv4) -- (nv5);
        \draw [thick,->] (nv5) -- (nv2);
      \end{scope}


      \begin{scope}[xshift=4cm]
        \coordinate (v1) at (0,0);
        \coordinate (v2) at (1,0);
        \coordinate (v3) at (0,1);
        \coordinate (v4) at (3,0);
        \coordinate (v5) at (2,0);
        \coordinate (v6) at (3,1);
        
        \fill [black!30] (v1) -- (v2) -- (v3);
        \fill [black!30] (v4) -- (v5) -- (v6);
        
        \foreach \i in {1,...,6}{
          \fill (v\i) circle (0.07cm);
          \node (nv\i) at (v\i) {};
        }
        \draw [thick,->] (nv1) -- (nv2);
        \draw [thick,->] (nv2) -- (nv3);
        \draw [thick,->] (nv6) -- (nv5);
        \draw [thick,->] (nv5) -- (nv4);
        
        \node at (1.5,-0.75) {$\Downarrow$};
      \end{scope}


      \begin{scope}[xshift=4cm,yshift=-2cm]
        \coordinate (v1) at (0,0);
        \coordinate (v2) at (1,0);
        \coordinate (v3) at (0,1);
        \coordinate (v4) at (3,0);
        \coordinate (v5) at (2,0);
        \coordinate (v6) at (3,1);
        
        \fill [black!20] (v1) -- (v2) -- (v3);
        \fill [black!20] (v4) -- (v5) -- (v6);
        
        \foreach \i in {1,3,4,6}{
          \fill [black!50] (v\i) circle (0.07cm);
          \node (nv\i) at (v\i) {};
        }
        \foreach \i in {2,5}{
          \fill (v\i) circle (0.07cm);
          \node (nv\i) at (v\i) {};
        }
        \draw [thick,->,black!50] (nv1) -- (nv2);
        \draw [thick,->,black!50] (nv2) -- (nv3);
        \draw [thick,->,black!50] (nv6) -- (nv5);
        \draw [thick,->,black!50] (nv5) -- (nv4);

        \draw [thick,->]
        ($(nv2.south east)!0.5!(nv2.east)$) --
        ($(nv5.south west)!0.5!(nv5.west)$);
        \draw [thick,->]
        ($(nv5.north west)!0.5!(nv5.west)$) --
        ($(nv2.north east)!0.5!(nv2.east)$);
      \end{scope}
      
    \end{tikzpicture}
    \caption{Resolving violations 1 (left) and 2 (right).}
    \label{fig:Conditions}
  \end{center}
  
\end{figure}
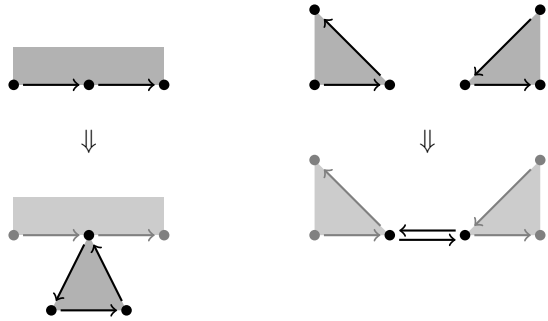

Suppose that violation 1 holds: there exists $\tuple{P, X} \in Q_p(C)$ and $w = (i, t) \in X$ that has no type in any polygon of $Q_p(C)$ it belongs to.
In particular, $(i, t-1)$ does not contain an error and $\eta^{t-1}_{i + j} \notin S^+_{n+1}$ for all $-r \leq j \leq r$.
Since $U_n$ and $V_n$ are $a_n, a_{n+1}$-forcing sets and $\eta^t_i \in S_n$, there exist $j_L \in U_n$ and $j_R \in V_n$ such that $\eta^{t-i}_{i + j_L}, \eta^{t-1}_{i + j_R} \in S_n$.
We also have $(j_L, -1) \in \Delta^R_n$, $(j_R, 1) \in \Delta^L_n$ and $(j_R - j_L, 0) \in \Delta^C_n$.
Denote $X' = \tuple{w, (i+j_L,t-1), (i+j_R,t-1)}$, and let $P' \subset \R^2$ be the triangle spanned by these three points.
Then $\tuple{P', X'}$ is a (possibly incomplete) three-vertex polygon and $w$ has type 3 in it.
We define $Q_{p+1}(C) = Q_p(C) \cup \{\tuple{P', X'}\}$, and $w$ is no longer a witness for violation 1.

Suppose then that violation 2 holds, so that some polygon $\tuple{P, X} \in Q_p(C)$ violates condition (d): some $v = (i,t), w = (j,t) \in X$ satisfy $0 < i - j \leq 2r$.
Thus we have $v - w \in \Delta^C_n$ and $w - v \in \Delta^B_n$.
Let $P' = \overline{v w}$ and $X' = \tuple{v, w}$.
Then $\tuple{P', X'}$ is a polygon with two vertices that satisfies every condition of Definition~\ref{def:Polygon} except the latter part of (e).
In particular, the pair $v, w$ does not violate condition (d) in $\tuple{P', X'}$.
We define $Q_{p+1}(C) = Q_p(C) \cup \{\tuple{P', X'}\}$, and then $v, w$ is no longer a witness for violation 2.

Suppose finally that violation 3 holds: there exist $\tuple{P, X}, \tuple{P', X'} \in Q_p(C)$ such that $P \cap P' \neq \emptyset$.
We construct a new polygon that contains the union $P \cup P'$.
Let $\tilde P$ be the polygon obtained by filling all holes in $P \cup P'$, and let $\tilde X = \tuple{w_1, \ldots, w_z}$ be the list obtained by traversing the border of $\tilde P$ in the counterclockwise direction and enumerating the elements of $X \cup X'$ in the order they are encountered.
Finally, let $\hat P$ be the simply connected polygon whose border is exactly the union of the line segments  $\overline{w_i w_{i+1}}$.
We claim that $(\hat P, \tilde X)$ is a valid space-time polygon.
It is easy to see that $w_i \in W \cap \partial \hat P$ for all $i \in \{1, \ldots, z\}$, and we now show that the differences of successive elements of the list $\tilde X$ have the correct form.

\begin{lemma}
\label{lem:GoodBorder}
For all $i \in \{1, \ldots, z\}$ we have $w_{i+1} - w_i \in \Delta^L_n \cup \Delta^R_n \cup \Delta^C_n \cup \Delta^B_n$.
\end{lemma}

\begin{proof}
We first show that the vertical distance between $w_i = (j, s)$ and $w_{i+1} = (j', s')$ is at most $1$, which implies that it is either $0$ or $1$.
Namely, there is a path from $w_i$ to $w_{i+1}$ along the border of $\tilde P$, and assuming $s' \geq s + 1$ (the case $s \geq s' + 1$ being analogous), this path runs through some coordinate $(j'', s + 1)$ with $j'' \in \R$; let this be the first such coordinate along the path.
Then $(j'', s + 1)$ is the endpoint of some line segment formed by consecutive elements of $X$ or $X'$, so in particular we have $(j'', s+1) \in X \cup X'$, which implies $(j'', s + 1) = w_{i+1}$.

Suppose now that $s' = s + 1$.
If the path $L$ from $w_i$ to $w_{i+1}$ along $\partial \tilde P$ is a single line segment, then it is contained in the border of either $P$ or $P'$, and we are done.
Otherwise, it is formed from two segments, one from $P$ and one from $P'$, and there are coordinates $j_1, j_2 \in \Z$ such that $(j_1, s)$ and $w_{i+1}$ are consecutive elements in one of the lists (say $X$), and $w_i$ and $(j_2, s+1)$ are consecutive elements in the other list (say $X'$), and $L$ is formed form parts of the associated line segments, which have a single crossing point.
See Figure~\ref{fig:VerticalBorder}.
Furthermore, we are traversing the border of $P \cup P'$ in the counterclockwise direction, which means that $(j_1,s)$ is to the left of $w_i$ and $(j_2,s+1)$ is to the left of $w_{i+1}$.
Consequently, $j_1 \leq j$ and $j_2 \leq j'$.
Since $(j_2,s+1) - w_i$ and $w_{i+1} - (j_1,s)$ are elements of $\Delta^R_n$, we have
\[ -r \leq j_2 - j \leq j' - j \leq j' - j_1 \leq -v_n \]
which implies $w_{i+1} - w_i \in \Delta^R_n$.
The case for $s' = s - 1$ is similar, but with $\Delta^L_n$ in place of $\Delta^R_n$.

\begin{figure}[ht]
  \begin{center}
    \begin{tikzpicture}

      \coordinate (wi) at (0,0);
      \coordinate (wii) at (-2,3);
      \coordinate (j1) at (-3,0);
      \coordinate (j2) at (-5,3);

      \path [name path=wij2] (wi) -- (j2);
      \path [name path=wiij1] (wii) -- (j1);
      \path [name intersections={of=wij2 and wiij1,by=int}];
      \fill [black!20] (wi) -- (int) -- (wii) -- (j2) -- (j1);

      \foreach \name in {wi,wii,j1,j2}{
        \node (n\name) at (\name) {};
      }
      \foreach \name in {wi,wii,j1,j2}{
        \fill (\name) circle (0.07);
      }
      
      \draw [thick,->] (nwi) -- (nwii);
      \draw [dotted] (nj1) -- (int);
      \draw [thick,dotted,->] (int) -- (nwii);
      \draw [thick,dotted] (nwi) -- (int);
      \draw [dotted,->] (int) -- (nj2);

      \node [right=0.2cm] at (wi) {$w_i$};
      \node [right=0.2cm] at (wii) {$w_{i+1}$};
      \node [left=0.2cm] at (j1) {$(j_1,s)$};
      \node [left=0.2cm] at (j2) {$(j_2,s+1)$};
      \node at ($(int)+(0.4,0.2)$) {$L$};
      
    \end{tikzpicture}
  \end{center}
  \caption{A visualization of the case $s' = s + 1$ in the proof of Lemma~\ref{lem:GoodBorder}. The shaded area is part of the union of $P \cup P'$.}
  \label{fig:VerticalBorder}
\end{figure}
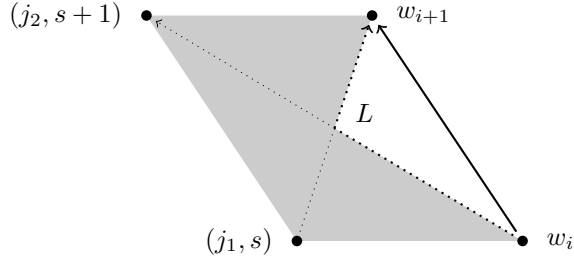

Finally, suppose $s' = s$.
Then there is a path $\gamma$ in $\partial P$ from $w_i$ to $w_{i+1}$ that either is a line segment, or consists of finitely many line segments and crosses an integral y-coordinate only at its endpoints.
In the former case, $\gamma$ is contained in a line segment between consecutive elements of either $X$ or $X'$, which implies $|j - j'| \leq 2 r$ and thus $w_{i+1} - w_i \in \Delta^C_n \cup \Delta^B_n$.
Suppose then that the latter case holds, and let $w_i = \hat w_1, \hat w_2, \ldots, \hat w_\ell = w_{i+1}$ be the endpoints of the line segments that form $\gamma$; equivalently, they are the vertices of $\tilde P$ that lie on $\gamma$.
We may assume that the y-coordinates of the $\hat w_p$ are between $s-1$ and $s$, and that $j' < j$.
The cases of the y-coordinates lying between $s$ and $s+1$ and/or $j' > j$ are symmetric, and in fact cannot actually happen in space-time polygons.

Each line segment $L_p$ from $\hat w_p$ to $\hat w_{p+1}$ is part of a longer segment between some $(j_p, s)$ and $(j'_p, s-1)$ in $\partial P$ or $\partial P'$.
We say that the segment $L_p$ is decreasing if $(j_p, s)$ is closer to $\hat w_p$ than $\hat w_{p+1}$, and increasing otherwise.
The interior angle $\alpha_p$ of $\tilde P$ at any endpoint $\hat w_p$ for $1 < p < \ell$ is greater than $\pi$, for otherwise $\hat w_p$ would be a vertex of $P$ or $P'$, and would lie in $X \cup X'$.
This implies two things.
First, there exists $1 \leq p' < \ell$ such that $L_p$ is decreasing for each $p \leq p'$, and increasing for each $p > p'$.
Second, $j_p < j_{p+1}$ holds whenever $1 \leq p < \ell$ and $p \neq p'$.
Since $j_1 = j$ and $j_\ell = j'$, this means $j_p \geq j$ for $p \leq p'$, and $j_p \leq j'$ for $p > p'$.
Each $w_p$ also has y-coordinate strictly above $s-1$, and hence we have $j'_{p'} < j'_{p'+1}$.
We now compute $|j - j'| \leq |j_{p'} - j'_{p'}| + |j_{p'+1} - j'_{p'+1}| \leq 2 r$, so that $w_{i+1} - w_i \in \Delta^C_n \cup \Delta^B_n$.
See Figure~\ref{fig:SegmentPath} for a visualization of this argument.
\end{proof}

\begin{figure}[ht]
  \begin{center}
    \begin{tikzpicture}[xscale=-1,scale=1.1]
      
      \coordinate (v1) at (0,0);
      \coordinate (v2) at (5,0);
      
      \coordinate (i1) at (v1);
      \coordinate (j1) at (-.5,-2);
      \coordinate (i2) at (-.75,0);
      \coordinate (j2) at (2,-2);
      \coordinate (i3) at (-2.5,0);
      \coordinate (j3) at (3,-2);
      \coordinate (i4) at (7,0);
      \coordinate (j4) at (1.5,-2);
      \coordinate (i5) at (v2);
      \coordinate (j5) at (4,-2);
      
      \coordinate (w2) at (intersection of i1--j1 and i2--j2);
      \coordinate (w3) at (intersection of i2--j2 and i3--j3);
      \coordinate (w4) at (intersection of i3--j3 and i4--j4);
      \coordinate (w5) at (intersection of i4--j4 and i5--j5);

      \draw ($(w5)+(60:0.2)$) arc (60:-158.2:0.2);
      \node [below left=0.08cm] at (w5) {$\alpha_5$};
      
      \fill (v1) circle (0.1cm);
      \fill (v2) circle (0.1cm);
      \draw [dotted] (-3,0) -- (7.5,0);
      \draw [dotted] (-3,-2) -- (7,-2);
      
      \foreach \n in {1,2,3,4,5}{
	\fill (i\n) circle (0.05cm);
	\fill (j\n) circle (0.05cm);
	\draw [dashed] (i\n) -- (j\n);
      }
      
      \path [draw,thick] (v1) -- (w2) -- (w3) -- (w4) -- node [above left] () {$\gamma$} (w5) -- (v2);
      
      \node [above] at (v1) {$w_i = \hat w_1$};
      \node [above left=-0.1 and 0.1] at (w2) {$\hat w_2$};
      \node [above left] at (w3) {$\hat w_3$};
      \node [above=.15cm] at (w4) {$\hat w_4$};
      \node [above right] at (w5) {$\hat w_5$};
      \node [above] at (v2) {$w_{i+1} = \hat w_6$};
      \node [left] at (7.5,0) {$s$};
      \node [left] at (7,-2) {$s-1$};
      \node [below right] at (j4) {$j'_{p'+1}$};
      \node [below left] at (j3) {$j'_{p'}$};
      \node [above] at (i4) {$j_{p'+1}$};
      \node [above right] at (i3) {$j_{p'}$};
      
    \end{tikzpicture}
  \end{center}
  \caption{The path $\gamma$ in the proof of Lemma~\ref{lem:GoodBorder}}
  \label{fig:SegmentPath}
\end{figure}
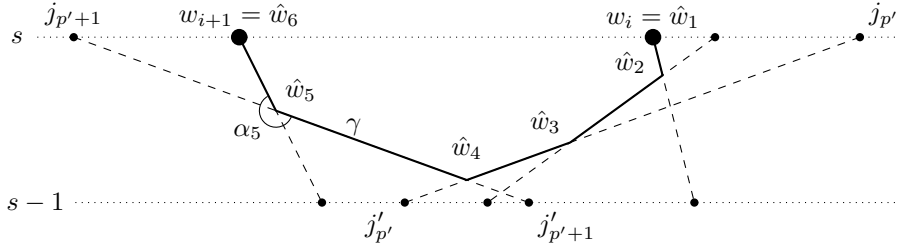

To finish the treatment of violation 3, we define $Q_{p+1}(C) = (Q_p(C) \setminus \{\tuple{P, X}, \tuple{P', X'}\}) \cup \{ \tuple{\hat P, \tilde X} \}$.

Now, all three constructions detailed above have the property that the union of all polygons in $Q_{p+1}(C)$ contains the union of all polygons in $Q_p(C)$.
In particular, if a vertex or pair of vertices witnesses violation 1 or 2 and we apply the construction to them, they cannot witness the violation again at any later stage.
Since there are finitely many sets of polygons on the vertex set $W$, we eventually reach a set $Q_p(C)$ that avoids violations 1, 2 and 3, and choose $Q(C) = Q_p(C)$.
As claimed, it is a set of disjoint space-time polygons of level $n$ whose union contains $C$.
At this point we also remark that in any polygon $\tuple{P, X} \in Q(C)$, any point $w \in X$ whose time coordinate is maximal is an element of $C$.
In particular, $X \cap C \neq \emptyset$.
However, note also that $C$ is not necessarily a subset of $X$, since some of its elements may cease to be vertices as polygons are merged, and become interior points instead.

We will now use the construction of space-time polygons to bound the probability of a given coordinate $(0, T)$ having a nonzero state in the random trajectory $\eta$.
For this purpose, suppose that $\eta^T_0 > 0$, and let $n_0 \in \{1, \ldots, k-1\}$ be such that $\eta^T_0 \in S_{n_0}$.
Define the initial set of vertices as $C_{n_0} = \{(0, T)\}$, and construct the set of level-$n_0$ space-time polygons $Q(C_{n_0})$.
Recall that each vertex $w = (i, t)$ that has type 1 in some polygon of $Q(C_{n_0})$ has a support point $\mathrm{supp}(w) = (j, t-1)$ at some level $n > n_0$ with $|i - j| \leq n r$.
Inductively, for each $n_0 < n \leq k$, we define $C_n \subset W$ as the set of level-$n$ support points of the polygons in $\bigcup_{p < n} Q(C_p)$:
\[
  C_n = \{ (i,t) \;|\; p < n, \tuple{P, X} \in Q(C_p), w \in X, (i,t) = \mathrm{supp}(w), \eta^t_i \in S_n \}
\]
This defines an indexed family $(Q(C_n))_{n = n_0}^k$ of sets of space-time polygons, where the polygons of each set $Q(C_n)$ are disjoint.

Fix two numbers $n_0 \leq n < \ell \leq k$, and let $\tuple{P, X} \in Q(C_\ell)$ and $\tuple{P', X'} \in Q(C_n)$ be two space-time polygons of different levels.
This means that $\tuple{P, X}$ is constructed at a later stage than $\tuple{P', X'}$, and its vertices have higher states.
The statements of Lemmas~\ref{lem:VertexDisjoint},~\ref{lem:BorderDisjoint} and~\ref{lem:NotInside} refer to these polygons.
The series of results shows that the two polygons can intersect only if $P \subset P'$.
First, we show that the vertex lists of the polygons are separated horizontally by at least $r$ steps.

\begin{lemma}
  \label{lem:VertexDisjoint}
  Let $w = (i, t) \in X$ be arbitrary.
  Then there are no elements $w' = (j, t) \in X'$ with $|i - j| \leq r$.
\end{lemma}

\begin{proof}
  Suppose on the contrary that such an element exists, and consider the first point in the construction of $Q(C_n)$ where $w'$ is added to the vertex set of some polygon.
  If $w'$ was introduced as a parent of another vertex $(j', t+1)$ that was a witness to violation 1, then $|j - j'| \leq r$, which implies $|i - j'| \leq 2 r \leq \ell r$.
    Then $(j', t+1)$ has type 1 in any polygon it belongs to, since $w$ is its potential support point, which contradicts the assumption that it witnessed violation 1 by having no type.
  Thus $w'$ was not introduced as a parent of another vertex.
  Since this is the only way of adding new elements to the vertex lists of polygons, the vertex $w'$ must have been present in the first phase $Q_0(C_n)$, which implies $w' \in C_n$.

  The initial vertex $(0, T)$ is the only vertex with time $T$, and the only element of $C_{n_0}$.
  Since $w$ and $w'$ have the same y-coordinate, this implies $w' \neq (0, T)$ and $n > n_0$.
  Thus the vertex $w'$ is the level-$n$ support point of some vertex $v = (j', t+1)$ with $\eta^{t+1}_{j'} \in S_{n'}$, $n' < n$ and $|j - j'| \leq n r$.
  But then $|i - j'| \leq |i - j| + |j - j'| \leq (n+1) r \leq \ell r$, and $w$ is also a potential support point for $v$.
  Since $\ell > n$ and the support point is chosen in a way that maximizes its level, the vertex $w'$ cannot be the support point of $v$, which contradicts $w' \in C_n$.
\end{proof}

\begin{lemma}
  \label{lem:BorderDisjoint}
  $\partial P \cap \partial P' = \emptyset$
\end{lemma}

\begin{proof}
  Assume for contradiction that $\partial P$ intersects $\partial P'$.
  Then there are consecutive vertices $v, w \in X$ and $v', w' \in X'$ such that the line segments $\overline{v w} \subset P$ and $\overline{v' w'} \subset P'$ intersect.
  The vertical distance between the endpoints of each segment is at most $1$, so some pair of endpoints, say $v$ and $v'$, have the same y-coordinate, say $v = (i, t)$ and $v' = (j, t)$ with $i \neq j$.
  If the segment $\overline{v w}$ is horizontal, then one of $v$ or $w$ is within distance $r$ of $v'$, contradicting Lemma~\ref{lem:VertexDisjoint}, and analogously if $\overline{v' w'}$ is horizontal.
  Thus both segments are non-horizontal, and we may assume $w = (i', t+1)$ and $w' = (j', t+1)$ with $i' \neq j'$.
  Since the segments intersect, we have either $i < j$ and $i' > j'$, or $i > j$ and $i' < j'$.
  In both cases, $|i - i'| \leq r$ and $|j - j'| \leq r$ imply $\min(|i - j|, |i' - j'|) \leq r$, again a contradiction with Lemma~\ref{lem:VertexDisjoint}.
\end{proof}

\begin{lemma}
  \label{lem:NotInside}
  There does not exist $w \in X'$ with $w \in P$.
\end{lemma}

\begin{proof}
  Suppose on the contrary that some $w \in X'$ satisfies $w \in P$.
  We may assume that $n$ is the minimal element of $\{n_0, \ldots, \ell-1\}$ that allows this.
  We have $w \in P \cap \partial P' \neq \emptyset$, so Lemma~\ref{lem:BorderDisjoint} implies $P' \subset P$.
  Let $v = (i,t) \in C_n \cap P'$ be arbitrary.
  Since $P' \subset P$, there exists a vertex of $X$ with the same y-coordinate as $v$.
  As $(0, T)$ is the only vertex of $C_{n_0}$ and the only vertex with y-coordinate equal to $T$, we have $v \neq (0, T)$ and $n > n_0$.
  Then there exists $m < n$, a polygon $\tuple{\hat P, \hat X} \in Q(C_m)$ and a vertex $u = (i', t+1) \in \hat X$ with $v = \mathrm{supp}(u)$.
  By the minimality of $n$, we have $u \notin P$.
  Then the line segment $I = \overline{v u}$ intersects $\partial P$, so there are consecutive elements $w_1, w_2 \in X$ such that the segment $\overline{w_1 w_2}$ intersects $I$.
  See Figure~\ref{fig:NotInside}.
  
  As in the proof of Lemma~\ref{lem:BorderDisjoint}, we can assume $w_1 = (j, t)$ and $w_2 = (j', t+1)$ with $i \neq i'$ and $j \neq j'$, and we have either $i < j$ and $i' > j'$, or $i > j$ and $i' < j'$.
  We also have $|i - i'| \leq n r$ and $|j - j'| \leq r$, and a simple calculation shows $|i' - j| \leq (n+1)r \leq \ell r$.
  Thus $w_1$ is also a potential support point for $u$ with a higher level, which contradicts $v = \mathrm{supp}(u)$.
\end{proof}

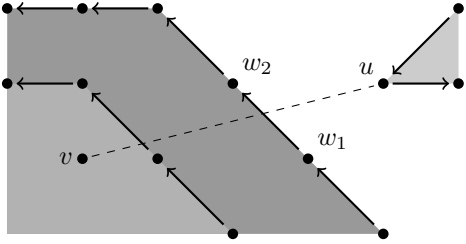
\begin{figure}[ht]
  \begin{center}
    \begin{tikzpicture}

      \coordinate (v1) at (0,0); 
      
      \coordinate (v2) at (2,-1);
      \coordinate (v3) at (1,0);
      \coordinate (v4) at (0,1);
      \coordinate (v5) at (-1,1);

      \coordinate (v6) at (4,-1);
      \coordinate (v7) at (3,0); 
      \coordinate (v8) at (2,1); 
      \coordinate (v9) at (1,2);
      \coordinate (v10) at (0,2);
      \coordinate (v11) at (-1,2);
      
      \coordinate (v12) at (5,2);
      \coordinate (v13) at (4,1); 
      \coordinate (v14) at (5,1);

      \fill [black!40] (v6) -- (v9) -- (v11) -- (-1,-1);
      \fill [black!30] (v2) -- (v4) -- (v5) -- (-1,-1);
      \fill [black!20] (v12) -- (v13) -- (v14);

      \foreach \i in {1,...,14}{
        \fill (v\i) circle (0.07cm);
        \node (n\i) at (v\i) {};
      }

      \foreach \a/\b in {2/3,3/4,4/5,6/7,7/8,8/9,9/10,10/11,12/13,13/14}{
        \draw [thick,->] (n\a) -- (n\b);
      }
      \draw [dashed] (n1) -- (n13);

      \node [left] at (v1) {$v$};
      \node [above left] at (v13) {$u$};
      \node [above right] at (v7) {$w_1$};
      \node [above right] at (v8) {$w_2$};
      
    \end{tikzpicture}
  \end{center}
  \caption{The vertices in the proof of Lemma~\ref{lem:NotInside}.}
  \label{fig:NotInside}
\end{figure}

Next, we show that every element of each set of support points $C_\ell$ is close to the border of the polygon that contains it.
We will use this fact to show that the total perimeter of all polygons in $Q(C_\ell)$ is bounded from below by a constant multiple of $|C_\ell|$.

\begin{lemma}
  \label{lem:CloseToBorder}
  Let $n_0 \leq \ell \leq k$, let $w = (i, t) \in C_\ell$, and let $\tuple{P, X} \in Q(C_\ell)$ be the space-time polygon containing $w$.
  Then there exists $w' = (i', t') \in X$ such that $|i - i'| \leq \ell r$ and $|t - s| \leq 1$.
\end{lemma}

\begin{proof}
  The case of $\ell = n_0$ is clear since $(0, T)$ is the only vertex of $C_{n_0}$ and the only vertex with y-coordinate at least $T$.
  Suppose thus that $\ell > n_0$.
  Since $w \in C_\ell$, there exists $n < \ell$, a polygon $\tuple{P', X'} \in Q(C_n)$ and a vertex $v = (j, t+1) \in X'$ such that $\eta^{t+1}_j \in S_n$ and $\mathrm{supp}(v) = w$.
  In particular, we have $|i - j| \leq \ell r$.
  Lemma~\ref{lem:NotInside} implies $v \notin P$.
  Then the line segment $\overline{w v}$, which is of length at most $\ell r + 1$, intersects some segment of the border $\partial P$.
  One endpoint $w' = (i', t') \in X$ of that segment satisfies $|i - i'| \leq \ell r$ and $|t - t'| \leq 1$, and the claim holds.
\end{proof}

We now show that out of all vertices in the system $(Q(C_n))_{n=n_0}^k$, a positive fraction have type 2.
Recall that we showed already in Lemma~\ref{lem:BorderSize} that for each individual space-time polygon, a positive fraction of its vertices have type 1 or 2.
The idea is that since type 1 vertices have support points, they give rise to new polygons of higher levels, and the polygons of level $k$, the maximum, have no type 1 vertices.

\begin{lemma}
  \label{lem:LinearBound}
  There exists $\beta > 0$, depending only on the automaton $f$, with the following property.
  Let $\mathcal{X} = \bigcup_{n = n_0}^k \bigcup_{\tuple{P,X} \in Q(C_n)} X$ be the set of all border vertices in the system of polygons starting from the initial vertex $C_{n_0} = \{(0, T)\}$.
  Then $|\{ w \in \mathcal{X} \;|\; \mbox{$w$ has type 2} \}| \geq \beta |\mathcal{X}|$.
\end{lemma}

\begin{proof} 
  For $n_0 \leq n \leq k$, denote $\mathcal{X}^a_n = \{ (i, t) \in \mathcal{X} \;|\; \mbox{$(i,t)$ has type $a$ and $\eta^t_i \in S_n$} \}$ and $\mathcal{X}_n = \bigcup_{a \in \{1,2,3\}} \mathcal{X}^a_n$.
  We define a process where each set $\mathcal{X}_n$ is given a non-negative weight $D(n)$, and the weights are iteratively re-distributed in a way that preserves their sum, which is precisely the number of vertices of type 2.
  We prove a positive lower bound on $D(n) / |\mathcal{X}_n|$ for each $n$ at the end of the process, which implies the claim of the Lemma.
  
  More formally, for all $n_0 \leq \ell \leq k+1$, we define a weight distribution $D_\ell : \{n_0, \ldots, k\} \to \R$ with $\sum_{n = n_0}^k D_\ell(n) = |\{ w \in \mathcal{X} \;|\; \mbox{$w$ has type 2} \}|$.
  The initial distribution $D_{k+1}$ is defined by $D_{k+1}(n) = |\mathcal{X}^2_n|$.
  Denote $\beta_{k+1} = 1$ and $\beta_n = \delta_n \beta_{n+1} / (2 r k + 2) < \beta_{n+1}$ for $n \geq k$, where $\delta_n > 0$ is given by Lemma~\ref{lem:BorderSize}.
  In the course of the construction, we maintain the following invariants:
  \begin{itemize}
  \item
    For each $n_0 \leq \ell \leq n \leq k$, we have $D_\ell(n) \geq \beta_n |\mathcal{X}_n|$.
  \item
    For each $n_0 < \ell \leq k + 1$, we have $D_\ell(\ell-1) \geq \delta_{\ell-1} \beta_\ell |\mathcal{X}_{\ell-1}|$.
  \end{itemize}
  
  For $\ell = k, k-1, \ldots, n_0$, we compute the new distribution $D_\ell$ from $D_{\ell+1}$ as follows.
  For each $p < \ell$ and each pair $(v, w) \in \mathcal{X}_p \times C_\ell$ such that $w$ is the support point of $v$, we transfer $\beta_\ell$ units of weight from $\mathcal{X}_\ell$ to $\mathcal{X}_p$.
  The resulting weight distribution is $D_\ell$.
  
  We now show that the invariants hold for $D_\ell$, and for the first one, let $n_0 \leq \ell \leq n \leq k$.
  If $\ell < n$, then $D_\ell(n) = D_{\ell+1}(n) \geq \beta_n |\mathcal{X}_n|$ by the induction hypothesis and the fact that $\mathcal{X}_n$ retains its weight in the construction of $D_\ell$.
  Suppose then that $\ell = n$.
  By Lemma~\ref{lem:CloseToBorder}, each $w \in C_\ell$ is within distance $r k$ from some vertex of $\mathcal{X}_\ell$ with the same y-coordinate, implying $|C_\ell| \leq (2 k r + 2) |\mathcal{X}_\ell|$.
  Since each such $w$ can be the support point for at most $2 r k + 1$ other vertices, at most $\beta_\ell (2 r k + 2)(2 r k + 1) = \delta_\ell \beta_{n+1} \frac{2 r k + 1}{2 r k + 2}$ units of weight were transferred from $\mathcal{X}_\ell$ via $C_\ell$.
  On the other hand, we have $D_{\ell+1}(\ell) \geq \delta_\ell \beta_{\ell+1} |\mathcal{X}_\ell|$ by the induction hypothesis.
  This implies
  \[
    D_\ell(\ell) \geq D_{\ell+1}(\ell) - |\mathcal{X}_\ell| \delta_\ell \beta_{\ell+1} \frac{2 r k + 1}{2 r k + 2} \geq \beta_\ell |\mathcal{X}_\ell|
  \]
  
  For the second invariant, let $n_0 < \ell \leq k+1$ and consider the weight $D_\ell(\ell-1)$.
  In the case $\ell = k+1$, we have $D_{k+1}(k) = |\mathcal{X}^2_k| \geq \delta_k \beta_{k+1} |\mathcal{X}_k|$ by Lemma~\ref{lem:BorderSize}, since $\mathcal{X}^1_k$ is empty and $\beta_{k+1} = 1$ by definition.
  Suppose then that $\ell \leq k$.
  For each type-1 vertex $w \in \mathcal{X}_{\ell-1}$ with support point of level $n \geq \ell$, exactly $\beta_n \geq \beta_\ell$ units of weight have been transferred to $\mathcal{X}_{\ell-1}$.
  In addition, it has the initial weight of $|\mathcal{X}^2_{\ell-1}|$ given by the type-2 vertices.
  Since $|\mathcal{X}^1_{\ell-1}| + |\mathcal{X}^2_{\ell-1}| \geq \delta_{\ell-1} |\mathcal{X}_{\ell-1}|$ by Lemma~\ref{lem:BorderSize}, this implies
  \[
    D_\ell(\ell-1) \geq \beta_\ell |\mathcal{X}^1_{\ell-1}| + |\mathcal{X}^2_{\ell-1}| > \delta_{\ell-1} \beta_\ell |\mathcal{X}_{\ell-1}|
  \]
  which is what we wanted to prove.
  
  All in all, we have shown that $D_{n_0}(\ell) \geq \beta_{n_0} |\mathcal{X}_\ell|$ holds for all $n_0 \leq \ell \leq k$.
  Since $\beta_{n_0}$ depends only on the automaton $f$, we can choose $\beta = \beta_{n_0}$, and the proof is complete.
\end{proof}

Next, we estimate the number of different space-time polygon systems rooted at the coordinate $(0, T)$ having a certain perimeter size $|\mathcal{X}| = H$.
Knowing this size, we can characterize the entire system $(Q(C_n))_{n = n_0}^k$ as follows.
First, for each polygon $\tuple{P, X} \in Q(C_n)$ in the system, there exists a vertex $w = (i,t) \in X \cap C_n$ (for example, one that maximizes $t$).
If $n > n_0$, there also exists $\ell < n$, another polygon $\tuple{P', X'} \in Q(C_\ell)$, and a vertex $w' = (j, t+1) \in X'$ such that $w = \mathrm{supp}(w')$ and $|i - j| \leq k r$.
We fix one such pair $(w, w')$ and call it the \emph{base} of $\tuple{P, X}$.
Now, we construct a list $L$ of vertices as follows.
We begin with the polygon $\tuple{P, X}$ containing the topmost vertex $w_1 = (0, T)$, and set $L_1 = X = \tuple{w_1, \ldots, w_{|X|}}$.

Suppose now that we have constructed a list $L_p$ for some $p \geq 1$.
If there is a polygon $\tuple{P', X'}$ that has not been processed yet, and its base $(w', w)$ has a vertex $w \in L_p$, then we replace $w$ in $L_p$ by the list of vertices $\tuple{w, w', w'_1, \ldots, w'_q, w', w}$ where $X' = \tuple{w', w'_1, \ldots, w'_q}$, and denote the resulting list by $L_{p+1}$.
If such a polygon does not exist, then the list $L_p = L = \tuple{w_1, \ldots, w_s}$ contains every vertex in $\mathcal{X}$ at least once and at most $2 k r$ times, so that $|L| \leq 2 k r H$.
Furthermore, for each index $i \in \{1, \ldots, s\}$ we have $w_{i+1} - w_i \in \{-k r, \ldots, k r\} \times \{-1, 0, 1\}$.
Since the list $L$, together with the mapping $L \to \{1, \ldots, k\}$ that sends each vertex $(i, t)$ to the number $n$ such that $\eta^t_i \in S_n$, characterizes $\mathcal{X}$ completely, the number of different systems $(Q(C_n))_{n = n_0}^k$ with $|\mathcal{X}| = H$ is at most $\sum_{s = H}^{2 k r H} (3 k (2 k r + 1))^s \leq (3 k (2 k r + 1))^{(2 k r + 1) H}$.

Now, if we have $\eta^T_0 > 0$, then there exists a system $\mathcal{P} = (Q(C_n))_{n = 1}^k$ of time-space polygons of some size $H > 0$ rooted at $(0, T)$; we denote this event by $W(\mathcal{P})$.
By Lemma~\ref{lem:LinearBound}, such a system contains at least $\beta H$ vertices of type 2, which correspond to a subset of the random error set $E$.
Since $\eta$ is generated by an $\epsilon$-perturbation of the cellular automaton $f$, the probability of $W(\mathcal{P})$ is at most $\epsilon^{\beta H}$.
Thus we have
\begin{align*}
  \prob[ \eta^T_0 > 0 ]
& \leq
  \sum_{H = 1}^\infty \sum_{\mathcal{P}} \prob[ W(\mathcal{P}) ]
\leq
  \sum_{H = 1}^\infty (3 k (2 k r + 1))^{(2 k r + 1) H} \epsilon^{\beta H} \\
& =
  \sum_{H = 1}^\infty \left( \frac{(3 k (2 k r + 1))^{2 k r + 1}}{\epsilon^\beta} \right)^H
\stackrel{\epsilon \to 0}{\longrightarrow}
  0
\end{align*}
since we have $\epsilon^\beta < (3 k (2 k r + 1))^{2 k r + 1}$ for all small enough $\epsilon$.
Then $f$ is a stable eroder, and we have proved the first half of Theorem~\ref{thm:StableEroder}.



\section{Stability Condition is Necessary}
\label{sec:Converse}

In this section, we prove the second half of Theorem~\ref{thm:StableEroder}: every one-dimensional stable eroder satisfies the stability condition.
The idea of the proof is the following.
We assume that a CA $f$ does not satisfy the stability condition, and our goal is to show that for all $\epsilon > 0$, there exists a finite island that $f$ almost surely never erodes, if each coordinate of the trajectory contains an error with probability of $\epsilon$ independently of the others.
Intuitively, the random errors extend the borders of the island faster than the automaton can erode it.
We begin by proving an alternative formulation of the stability condition for one-dimensional automata.

\begin{lemma}
  \label{lem:AltStability}
  The automaton $f$ satisfies the stability condition if and only if for all quiescent $a \in S \setminus \{0\}$ there exists quiescent $b < a$ such that $L_{b,a} > R_{a,b}$.
\end{lemma}

\begin{proof}
  Suppose that $f$ satisfies the latter condition.
  Denote $a_1 = m$.
  Then there exists quiescent $a_2 < a_1$ such that $L_{a_2,a_1} > R_{a_1,a_2}$.
  Iterating this argument at most $|S|-1$ times, we find a sequence $(a_i)_{i=1}^k$ of quiescent states such that $m = a_1 > a_2 > \cdots > a_k = 0$ and $L_{a_{i+1},a_i} > R_{a_i,a_{i+1}}$ for all $i \in \{1, \ldots,  k-1\}$.
  This is precisely the stability condition.
  
  Suppose then that $f$ satisfies the stability condition with the sequence $0 = a_1 < \cdots < a_k = m$, and let $a \in S$ be quiescent.
  There is a unique $i \in \{1, \ldots, k-1\}$ such that $a_i < a \leq a_{i+1}$.
  By the definition of $a_i$ and $a_{i+1}$, we have $L_{a_i,a_{i+1}} > R_{a_{i+1},a_i}$.
  From \eqref{eq:LIneq} we deduce $L_{a_i,a} \geq L_{a_1,a_{i+1}}$, and \eqref{eq:RIneq} gives $R_{a_{i+1},a_i} \geq R_{a,a_i}$.
  Thus we have $L_{a_i,a} > R_{a,a_i}$, which is what we wanted to show.
\end{proof}

Now, we define a special kind of $\epsilon$-perturbation of a cellular automaton: one where errors occur independently with probability $\epsilon$, and always produce a fixed state $a \in S$ unless it would cause the state to decrease.
Our goal is to prove that if a CA fails to satisfy the condition of Lemma~\ref{lem:AltStability}, then these perturbations form a counterexample to it being a stable eroder.

\begin{definition}
  Let $f$ be a cellular automaton on $S^\Z$, let $a \in S$, and let $\epsilon > 0$.
  The \emph{independent $a,\epsilon$-perturbation} of $f$ is the stochastic symbolic process $R$ defined as follows.
  Let $E \subset \Z \times \N$ be a random set, where each $v \in \Z \times \N$ belongs to $E$ independently with probability $\epsilon$.
  Then $R(x) = R_E(x)$ is defined by
  \[
    R(x)^{t+1}_i =
    \left\{
      \begin{array}{ll}
        \max(a, f(R(x)^t)_i), & \mbox{if~} (i, t) \in E. \\
        f(R(x)^t)_i, & \mbox{if~} (i, t) \notin E.
      \end{array}
    \right.
  \]
  The set $E$ is called the \emph{underlying error set}.
\end{definition}

Note that $R(x)^{t+1} \geq f(R(x)^t)$ holds for all $x \in S^\Z$ and $t \in \N$, since an error can only increase the value of a coordinate.

The following result is a stronger version of Proposition~\ref{prop:StableIsEroder}.
It tells us that in order to prove that a CA $f$ is not a stable eroder, it suffices to find a finite island that any given independent perturbation of $f$ never erodes away with arbitrarily high probability.
Note that we fix a sequence of coordinates that the evolving island must contain.
Also, we cannot require that an island survives forever with probability 1, since $f$ might be an eroder, and there is a small but positive probability that no errors occur in the vicinity of the island before it is eroded away.

\begin{lemma}[Proposition~3 of \cite{To76}]
  \label{lem:ImmortalIsland}
  Let $f$ be a one-dimensional monotonic CA, let $a \in S$, and let $R^a_\epsilon$ be the independent $a,\epsilon$-perturbation of $f$.
  Suppose that for all $N \in \N$, there exists a $0$-island $x^N \in S^\Z$ with $x^N \leq \uniform{a}$ and a sequence of coordinates $(i^N_t)_{t \in \N}$ such that for all $\epsilon > 0$, we have
  \[
    \inf_{t \in \N} \prob [R^a_\epsilon(x^N)^t_{i^N_t} = a] \stackrel{N \to \infty}{\longrightarrow} 1
  \]
  Then $f$ is not a stable eroder.
\end{lemma}

For the remainder of this section, we fix a monotonic automaton $f$ that does not satisfy the condition of Lemma~\ref{lem:AltStability}.
Then there exists a quiescent state $\omega \in S \setminus \{0\}$ such that $L_{a,\omega} \leq R_{\omega,a}$ for all quiescent $a < \omega$.
If there are several such states, let $\omega$ be the lowest one.
We fix a small $\epsilon > 0$, and let $R$ be the independent $\omega, \epsilon$-perturbation of $f$.
Note that a random trajectory $R(x)$ will not contain a state greater than $\omega$ if $x \leq \uniform{\omega}$, so we can safely forget the states $\{\omega + 1, \ldots, m\}$ and assume $S = \{0, \ldots, \omega\}$.

\begin{definition}
  \label{def:Inductive}
  Let $a < \omega$ be a quiescent state of $f$.
  We say that $a$ is \emph{inductive} if there exist positive real numbers $0 < \delta_a < 1$, $0 < \theta_a \leq r$ and $Q_a > 0$ with the following property for all large enough $N \geq 1$:
  If $x \in S^\Z$ is such that $x \geq \uniform{a}$ and $x_i \geq \omega$ for $-N \leq i \leq N$, then
  \[
    \prob \left[ \forall t \in \N, t(L_{a, \omega} - \theta_a) < i < t(R_{\omega,a} + \theta_a) :  R(x)^t_i = \omega \right] \geq 1 - \delta_a^{N^{Q_a}}
  \]
\end{definition}

The intuition for an inductive state $a$ is that with a high probability, every large enough island of $\omega$-states surrounded by $a$-states will spread at an average speed strictly greater than $L_{a,\omega} - R_{\omega,a}$.
This causes a ``cone'' of $\omega$-states to appear in the trajectory of the configuration.
Our goal in the remainder of this section is to prove Proposition~\ref{prop:Inductive}: every state $a < \omega$ is inductive.
In particular, the lowest state $0$ is inductive, and then we can apply Lemma~\ref{lem:ImmortalIsland} to the island configurations ${}^\infty 0 \omega^N . \omega^N 0^\infty$ to prove that $f$ is not a stable eroder.

\begin{proposition}
  \label{prop:Inductive}
  Every quiescent state $a < \omega$ is inductive.
\end{proposition}

We will actually prove the following one-sided versions of inductivity, which makes the arguments conceptually simpler.

\begin{lemma}
  \label{lem:InductiveOnesided}
  Let $0 \leq a < \omega$, $\theta > 0$, $0 < \delta < 1$ and $Q > 0$.
  Suppose that we have
  \begin{align}
    \label{eq:InductiveStrong}
    \prob[ \forall t \in \N, i < t(R_{\omega,a} + \theta_a) :  R(x)^t_i = \omega ] \geq 1 - \delta^{N^Q} \\
    \label{eq:InductiveStrong2}
    \prob[ \forall t \in \N, i > t(L_{a,\omega} - \theta_a) :  R(y)^t_i = \omega ] \geq 1 - \delta^{N^Q}
  \end{align}
  for $x = {}^\infty \omega . \omega^N a^\infty$ and $y = {}^\infty a \omega^N . \omega^\infty$ and all large enough $N$.
  Then $a$ is an inductive state.
\end{lemma}

\begin{proof}
  Let $K > 0$ be given by Lemma~\ref{lem:Galperin}.
  Denote $x = {}^\infty \omega . \omega^{2 N} a^\infty$ and $y = {}^\infty a \omega^{2 N} . \omega^\infty$.
  Consider the configuration $z = {}^\infty a \omega^{2 N} . \omega^{2 N} a^\infty$, and construct a coupling of the three random trajectories $R(x)$, $R(y)$ and $R(z)$ where the same underlying error set is used for each.
  Denote by $P$ the event that $R(y)^t_i = R(x)^t_i = \omega$ for all $t \in \N$ and $t (L_{a,\omega} - \theta_a) - N \leq i \leq t (R_{\omega,a} + \theta_a) + N$.
  From~\eqref{eq:InductiveStrong},~\eqref{eq:InductiveStrong2}, the assumption $L_{a,\omega} \leq R_{\omega,a}$ and the union bound, we obtain $\prob [P] \geq 1 - 2 \delta^{N^Q}$.
  If $P$ occurs and $N > r$, then for all $t \in \N$, the distance between the boundaries of $\omega$-states in $R(y)^t$ and $R(x)^t$ is at least $2 r$.
  This means that to the cellular automaton $f$, the trajectory $R(z)$ looks locally like $R(y)$ to the left of $t (R_{\omega,a} + \theta_a) + N$, and like $R(x)$ to the right of $t (L_{a,\omega} - \theta_a) - N$.
  More formally, we can prove by induction on $t$ that $R(z)^t_i = R(y)^t_i$ for all $i \leq t (R_{\omega,a} + \theta_a) + N$, and $R(z)^t_i = R(x)^t_i$ for all $i \geq (L_{a,\omega} - \theta_a) - N$.
  In particular, this shows that $R(z)^t_i = \omega$ for all $t \in \N$ and $t (L_{a,\omega} - \theta_a) \leq i \leq t (R_{\omega,a} + \theta_a)$.

  We choose $\delta_a = 2^{1/M} \delta$ where $M \geq 1$ is such that $\delta_a < 1$.
  By the above, these constants satisfy the conditions of Definition~\ref{def:Inductive}.
\end{proof}

Most of the remainder of this section is devoted to the proof of Proposition~\ref{prop:Inductive}.
It suffices to prove~\eqref{eq:InductiveStrong}, since the other case is symmetric.
We will use an auxiliary result about biased random walks.

\begin{lemma}
  \label{lem:RandWalkHitting}
  Let $(X_n)_{n \in \N}$ be a sequence of independent random variables with $\prob[X_n = 1] = p$ and $\prob[X_n = 0] = 1-p$, where $0 < p < 1$.
  For all $0 < q < p$, there exists $0 < d < 1$ such that
  \[ \prob \left[ \exists n \geq 0 : N + \sum_{i = 0}^n X_n \leq n q \right] \leq d^N \]
  for all $N \geq 1$.
\end{lemma}


We prove Proposition~\ref{prop:Inductive} by downward induction on the state $a$, and the idea of the proof is this.
We consider the configuration of~\eqref{eq:InductiveStrong} for a large $N \in \N$.
By Lemma~\ref{lem:BorderGuard}, there exists a quiescent state $a < b \leq \omega$ with $L_{b,a} = R_{\omega,a}$.
Either $b = \omega$ or $b$ is inductive.
In the former case, we essentially have a biased random walk on the $\omega,a$-interface of the configuration ${}^\infty \omega . \omega^N a^\infty$.
The random walk is slightly biased to the right of $R_{\omega,a}$, and we can apply Lemma~\ref{lem:RandWalkHitting} to it.
In the latter case, we maintain two regions in the random trajectory: an inner region of $\omega$-states, and an outer region of quiescent states at least $b$.
The surface of the outer region behaves like a random walk as in the first case, and the surface of the inner region is constantly ``repaired'' by randomly appearing patches of $\omega$-states that produce small $\omega$-cones as per the induction hypothesis.
Before proceeding with the proof, we formalize the argument about random walks, since it is used in both cases.
Recall the number $K > 0$ from Lemma~\ref{lem:Galperin}.

\begin{lemma}
  \label{lem:RandWalkInterface}
  Let $0 < d < 1$ be given by Lemma~\ref{lem:RandWalkHitting} for $p = \epsilon^{K+1}$ and $q = p/2$.
  Let $a < b \in S$ be quiescent states with $L_{b,a} = R_{b,a}$.
  Let $M \in \N$, and let $x = {}^\infty b . b^{M + K} a^\infty$ be the step of type $b,a$ shifted $M+K$ steps to the right.
  Then
  \[
    \prob [ \forall t \in \N, i \leq (R_{b,a} + q) t : R(x)^t_i \geq b ] \geq 1 - d^M
  \]
\end{lemma}

\begin{proof}
  We know that $f^t(x)_i = b$, and thus $R(x)^t_i \geq b$, holds for all $t \geq 0$ and $i < L_{b,a} t + M$.
  Let $E \subset \Z \times \N$ be the underlying error set of $R(x)$.
  Let us define a sequence of random variables $(X(t))_{t \in \N}$ with values in $\{0, 1\}$ as follows.
  For $t, s \in \N$, let $i^s_t \in \Z$ be the largest integer with $f^s(R(x)^t)_i \geq b$ for all $i \leq i^s_t$.
  Note that $i^0_{t+s} \geq i^s_t \geq i^0_t + L_{b,a} s - K$ for all $t, s \in \N$ (the latter inequality follows from Lemma~\ref{lem:Galperin}).
  We set $X(t) = 1$ if and only if $(i^1_t + j, t - 1) \in E$ holds for all $j \in \{1, \ldots, K+1\}$.
  In that case we have  $i^0_{t+1} \geq i^1_t + K + 1$, which implies $i^0_{t+s} \geq i^0_t + L_{b,a} s - K + 1$ for all $s \geq 1$.
  In general, if $1 \leq t_1 < t_2 < \cdots < t_n$ are such that $X(t_k) = 1$ for all $1 \leq k \leq n$, then $i^0_{t_n + s} \geq i^0_0 + L_{b,a} s - K + n = L_{b,a} s + M + n$ for all $s \geq 1$.
  
  Each $X(t)$ has value 1 independently with probability $p = \epsilon^{K+1}$, so the sum $\sum_{t = 1}^T X_t$ forms a simple random walk.
  By Lemma~\ref{lem:RandWalkHitting}, the probability that $i^0_t < (L_{b,a} + q) t$ for some $t \in \N$ is then at most $\alpha^M$.
  This is equivalent to $R(x)^t_i < b$ for some $t \in \N$ and $i \leq (L_{b,a} + q) t = (R_{b,a} + q) t$.
\end{proof}

We proceed with the proof of Proposition~\ref{prop:Inductive} by~\eqref{eq:InductiveStrong}.
Assume first that $b = \omega$, so that $L_{\omega,a} = R_{\omega,a}$.
When $N > K$, we can apply Lemma~\ref{lem:RandWalkInterface} to $M = N - K$, and choose $\theta_a = q = \epsilon^{K+1}/2$, $\delta_a = d$ and $Q_a = 1$ to obtain~\eqref{eq:InductiveStrong}.

Suppose now that  $a < b < \omega$, so that $b$ is inductive by the induction hypothesis.
Without loss of generality, we assume that $R_{\omega,a} > 0$.
If this is not the case, we compose $f$ with a large right shift, which does not change its eroding properties.
Define $x = {}^\infty \omega . \omega^N a^\infty$ and $\theta_a = \epsilon^{K+1} / 3$.
We will determine the values of $\delta$ and $Q$ later.

\begin{definition}
  We say that a trajectory $R(x)$ is \emph{$\omega$-good} if
  \begin{equation}
    \label{eq:InnerCone}
    \forall t \in \N, i \leq t (R_{\omega,a} + \theta_a) : R(x)^t_i = \omega
  \end{equation}
  and \emph{$b$-good} if
  \begin{equation}
    \label{eq:OuterCone}
    \forall t \in \N, i \leq t (R_{\omega,a} + 2 \theta_a) : R(x)^t_i \geq b
  \end{equation}
\end{definition}

Our goal in this proof is to maintain an inner region of $\omega$-states and a slightly wider outer region of states that are at least $b$.
The above definitions formalize this goal.
A trajectory is $\omega$-good if we are able to maintain the inner region indefinitely, and $b$-good if we maintain the outer region indefinitely.
As in the proof of the case $b = \omega$, maintaining the outer region indefinitely is relatively simple.

\begin{lemma}
  \label{lem:OuterRegion}
  Let $0 < d < 1$ be given by Lemma~\ref{lem:RandWalkHitting} for $p = \epsilon^{K+1}$ and $q = p / 2$.
  For $N > K$, the probability of $R(x)$ being $b$-good is at least $1 - d^{N - K}$.
\end{lemma}

\begin{proof}
  This follows directly from Lemma~\ref{lem:RandWalkInterface}.
\end{proof}

We now turn to the problem of maintaining the inner region.
For convenience, we make the following observation about the state $b$: in a trajectory that starts from a configuration above $\uniform{b}$ and contains a long sequence of $\omega$-states, with a high probability one finds a large rectangle of $\omega$-states.
This is simply because one expects to find an infinite cone of $\omega$-states by the inductive hypothesis of Proposition~\ref{prop:Inductive}, inside which the rectangle can be picked.
See Figure~\ref{fig:AlphaBetaCone} for a visualization.
Note that we have a lot of freedom in choosing the numbers $\alpha$ and $\beta$.
In particular, we can choose $\alpha$ as large as we want, and $\beta$ as small as we want.
Recall the definition $D_b = (L_{b,\omega} + R_{\omega,b})/2$.

\begin{observation}
  \label{obs:AlphaBeta}
  There exist numbers $\alpha > 0$, $0 < \beta < 1$ and $0 < \lambda < 1$ with the following property.
  Let $C, M \geq 1$ (not necessarily integers), and let $y \in S^\Z$ be such that $y_i = \omega$ for $0 \leq i \leq M$ and $y_i \geq b$ for $i \leq C \alpha$.
  If $M$ and $C / M$ are large enough, then
  \[
    \prob \left[ \forall \; 0 \leq i \leq (R_{\omega,a} + \theta_a) C \beta, 0 \leq t \leq 2 C \beta : R(y)^{\lfloor C + t \rfloor}_{\lfloor C D_b + i \rfloor} = \omega \right] \geq 1 - \lambda^{M^{Q_b}}
  \]
  If this event occurs, we call the set $[C D_b, C D_b + (R_{\omega,a} + \theta_a) C \beta] \times [C, C + 2 C \beta]$ an \emph{$\omega$-rectangle}.
\end{observation}

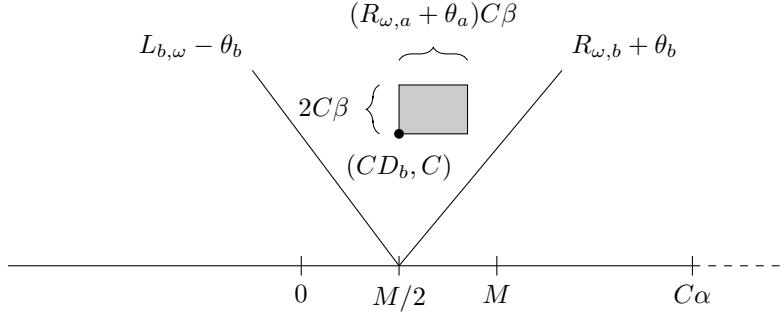
\begin{figure}[ht]
  \begin{center}
    \begin{tikzpicture}[xscale=1.3,yscale=1.3]

      \coordinate (bl) at (0,1.35);
      \coordinate (br) at (0.7,1.35);
      \coordinate (tl) at (0,1.85);
      \coordinate (tr) at (0.7,1.85);
      
      \draw [fill=black!20] (bl) rectangle (tr);

      \fill (bl) circle (0.05cm);
      
      \foreach \x in {-1,0,1,3}{
        \draw (\x,-.1) -- ++(0,.2);
      }
      
      \draw (0,0) -- (-1.5,2);
      \draw (0,0) -- (5/3,2);
      

      \coordinate (bline) at (-0.2,0);
      \coordinate (tline) at (-0.2,1);
      
      \draw [decorate,decoration={brace,amplitude=5pt}] (intersection of bline--tline and bl--br) -- node [midway,left,xshift=-8pt] {$2 C \beta$} (intersection of bline--tline and tl--tr);

      \coordinate (lline) at (-1,1.9);
      \coordinate (rline) at (1,1.9);

      \coordinate (lbr) at (lline);
      \coordinate (mbr) at (intersection of lline--rline and bl--tl);
      \coordinate (rbr) at (intersection of lline--rline and br--tr);
      
      \draw [decorate,decoration={brace,amplitude=5pt}] ($(mbr)+(0,.2)$) -- ($(rbr)+(0,.2)$) node [midway,above,yshift=8pt] {$(R_{\omega,a} + \theta_a) C \beta$};
      
      \draw (-4,0) -- (3,0);
      \draw [dashed] (3,0) -- (4,0);
      
      \node [below] at (-1,-.1) {$0$};
      \node [below] at (0,-.1) {$M/2$};
      \node [below] at (1,-.1) {$M$};
      \node [below] at (3,-.1) {$C \alpha$};

      \node [below=0.15cm] at (bl) {$(C D_b, C)$};
      
      \node [above left] at (-1.5,2) {$L_{b,\omega} - \theta_b$};
      \node [above right] at (5/3,2) {$R_{\omega,b} + \theta_b$};
      
      
    \end{tikzpicture}
  \end{center}
  \caption{Geometric interpretation of Observation~\ref{obs:AlphaBeta}, not drawn to scale. Time increases upward. With high probability, the cone contains only $\omega$-states by the induction hypothesis, so the shaded rectangle contains only $\omega$-states as well. The non-dashed part of the horizontal line contains states that are at least $b$, and the part between $0$ and $M$ contains $\omega$-states.}
  \label{fig:AlphaBetaCone}
\end{figure}

The constant $\alpha$ in the Observation should be chosen so that any coordinate of $y$ to the right of $C \alpha$ has no influence on the evolution of the $\omega$-rectangle.
Such a choice is possible due to the finite radius of $f$.
The rectangle is nonempty for large enough $C$, since we assumed $R_{a, \omega} > 0$.
We will now dispose of the requirement of $\omega$-states in Observation~\ref{obs:AlphaBeta} by considering a smaller $\omega$-rectangle.

\begin{lemma}
  \label{lem:Rectangle}
  Let $C \geq 1$, and let $y \in S^\Z$ satisfy $y_i \geq b$ for $i \leq C (\alpha + r \beta)$. 
  If $C$ is large enough, then
  \begin{align*}
    & \prob \left[ \forall \; 0 \leq i \leq (R_{\omega,a} + \theta_a) C \beta, C \beta \leq t \leq 2 C \beta : R(y)^{t + \lfloor C \rfloor}_{i + \lfloor C D_b \rfloor} = \omega \right] \\
    {} \geq {}
    & 1 - \lambda^{C^{Q_b/3}} 
  \end{align*}
  with the notation of Observation~\ref{obs:AlphaBeta}.
\end{lemma}

\begin{proof}
  The idea of the proof is to consider a family of $\sqrt{C \beta}$ rectangles that overlap in the desired region.
  Each of them can be generated by a short horizontal segment of $\omega$-states that happens to occur in the trajectory at the right position.
  Instead of applying Observation~\ref{obs:AlphaBeta} directly to these segments, we only consider those segments that grow at a linear pace for $\sqrt{C \beta}$ steps, which produces a better estimate for the probability of the final rectangle.

  Denote $u = \lfloor \sqrt{C \beta} \rfloor$, and for $0 \leq n < u$, denote $t_n = n u$.
  Let $M \geq 1$ be an integer constant so large that Observation~\ref{obs:AlphaBeta} holds for it, and $\gamma = \epsilon^{M + 1} + \delta_b^{M^{Q_b}} < 1$.
  Denote by $P_1(n)$ the event that $(i - \lceil D_b u \rceil, t_n-1) \in E$ for all $0 \leq i \leq M$, and by $P_2(n)$ the event that $R(y)^{t_{n+1}}_i = \omega$ for all $0 \leq i \leq (R_{\omega,a} + \theta_a) u \beta$.

  We now estimate the probabilities of $P_2(n)$ and the occurrence of the $\omega$-rectangle.
  We have $\prob [P_1(n)] = \epsilon^{M + 1}$ for each $0 \leq n < u-1$, since the elements of $E$ are chosen independently.
  If $P_1(n)$ occurs, then $R(y)^{t_n}_{i - \lceil D_b u \rceil} = \omega$ for each $0 \leq i \leq M$.
  Furthermore, we have
  \begin{equation}
    \label{eq:Spread}
    R(y)^t_i \geq b \text{~for all~} t \geq 0, i \leq C (\alpha + r \beta) - r t
  \end{equation}
  since $r$ is a radius for $f$.
  Since $t_n \leq C \beta$ and $b$ is a quiescent state, this implies $R(y)^{t_n}_{i - \lceil D_b u \rceil} \geq b$ for all $i \leq C \alpha + \lceil D_b u \rceil$.
  Noticing that $C \alpha + \lceil D_b u \rceil \geq u \alpha$, we now apply Observation~\ref{obs:AlphaBeta} to (a shifted version of) $R(y)^{t_n}$ with the variables $M$ and $u$.
  If $C$ is large enough, this gives us $R(y)^{t_{n+1}}_i = \omega$ for all $0 \leq i \leq (R_{\omega,a} + \theta_a) u \beta$ with probability at least $1 - \lambda^{M^{Q_b}}$.
  Hence we have $\prob [P_2(n) \;|\; P_1(n)] \geq 1 - \delta_b^{M^{Q_b}}$.
  This bound is valid even when conditioned on any combination of $P_1(k)$ and $P_2(k)$ for any $k < n$, since it only depends on the composition of the error set $E$ between time steps $t_n$ and $t_{n+1}$.
  As long as $C$ is large enough, the probability of $P_2(n)$ occurring for some $0 \leq n < u-1$ is thus at least $1 - (1 - \gamma)^{u-1}$.
  
  Next, let $P_3(n)$ denote the event that $P_2(n)$ holds but $P_2(k)$ does not hold for any $k < n$.
  If $P_3(n)$ holds, then~\eqref{eq:Spread} implies $R(y)^{t_{n+1}}_i \geq b$ for all $i \leq C \alpha$, since $t_{n+1} \leq C \beta$.
  We can then apply Observation~\ref{obs:AlphaBeta} to $R(y)^{t_{n+1}}$ with the variables $(R_{\omega,a} + \theta_a) u \beta$ and $C$.
  This yields
  \begin{align}
    \label{eq:RectBound}
    \begin{split}
    & \prob \left[ \forall \; 0 \leq i \leq (R_{\omega,a} + \theta_a) C \beta, 0 \leq s \leq 2 C \beta  : R(y)^{t_{n+1} + \lfloor C + s \rfloor}_{\lfloor C D_b + i \rfloor} = \omega \;\middle|\; P_3(n) \right] \\
    {} \geq {} & 1 - \lambda^{((R_{\omega,a} + \theta_a) u \beta)^{Q_b}}
    \end{split}
  \end{align}
  which implies the event in the statement of this Lemma, as $t_{n+1} \leq C \beta < C$.
  
  If $C$ is large enough, we have $\lambda^{C^{Q_b/3}} \geq (1 - \gamma)^{u-1} + \lambda^{((R_{\omega,a} + \theta_a) u \beta)^{Q_b}}$.
  The claim now follows from inequality~\eqref{eq:RectBound}, the union bound and the fact that the events $P_3(n)$ form a partition of the union of the events $P_2(n)$.
\end{proof}

In the case that the outer region can be maintained indefinitely, in the sense that the trajetory is $b$-good, the evolution of the inner region can be seen as a two-dimensional generalized bootstrap percolation process.
Initially, all cells $(i, 0)$ with $i \leq N$ are active.
Lemma~\ref{lem:Rectangle} implies the following for all large enough $C$.
If $i + C(\alpha + r \beta) \leq t(R_{\omega,a} + 2 \theta_a)$, so that all cells between $(i, t)$ and $(i+C(\alpha + r \beta), t)$ are within the outer region, then with a combined probability of at least $1 - \lambda^{C^{Q_b/3}}$, the cell $(i + j, t + s)$ is active for each $(j, s) \in [C D_b, C D_b + (R_{\omega,a} + \theta_a) C \beta] \times [C, C + C \beta]$.
These probabilities are independent for those choices of $(i,t)$ that are far enough from each other, depending on the respective choices of $C$, and all of them are positively correlated.
In this way, we obtain an initial distribution $A_0 \subset \Z \times \N$ of active cells in a two-dimensional random configuration in the form of a half-infinite line at time $0$ and a random collection of rectangles.
The set $A_{k+1}$ contains all cells of $A_k$, and each cell $(i, t)$ such that $(i+j, t-1) \in A_k$ for each $-r \leq j \leq r$.
As a limit of this percolation process, we obtain a set $A_\infty = \bigcup_{k = 0}^\infty A_k$ of active cells, with the property that $R(x)^t_i = \omega$ for all $(i, t) \in A_\infty$.
Thus, the conditional probability of $R(x)$ being $\omega$-good given that it is $b$-good is at least the probability of $(i, t) \in A_\infty$ for all $t \geq 0$ and $i \leq t (R_{\omega, a} + \theta_a)$.


The following lemma establishes a lower bound for the probability of maintaining the inner region indefinitely, expressed as a property of the set $A_\infty$.
The idea of the proof is to show that with high probability the rectangles in $A_0$, together with the cells activated by the initial half-line, cover the entire discrete line of coordinates $(i,t)$ with $i \approx t (R_{\omega, a} + \theta_a)$.

\begin{lemma}
\label{lem:AllGood}
There exists $\xi > 0$ with the following property.
Suppose that $R(x)$ is $b$-good.
If $N$ is large enough, the conditional probability of $R(x)$ being $\omega$-good is at least $1 - \xi^{Q_b/3}$. 
\end{lemma}

\begin{proof}
  Let $\alpha, \beta, \lambda$ be given by Observation~\ref{obs:AlphaBeta}.
  Denote
  \[
    \rho = \frac{\theta_a}{\alpha + r \beta - D_b + R_{\omega,a} + 2 \theta_a}
  \]
  We may choose $\alpha$ so large that $0 < \rho < 1$.
  We define a sequence of time steps by $t_k = \lceil \frac{N}{4 r} (1 + \rho \beta/2)^{k/2} \rceil$ for all $k \geq 0$.
  Consider one of these time steps $t_k$, and set $i_k = \lfloor t_k (R_{\omega, a} + \theta_a) \rfloor$.
  Denote also $C_k = \rho t_k$, $j_k = i_k - \lceil C_k D_b \rceil$ and $s_k = t_k - \lfloor C_k \rfloor$.
  If $\alpha$ and $N$ are large enough, we have $j_k, s_k \geq 1$ for all $k \geq 0$.
  Denote by $P(k)$ the event that $(i_k + i, t_k + t) \in A_0$ for all $0 \leq i \leq (R_{\omega,a} + \theta_a) C_k \beta$ and $0 \leq t \leq C_k \beta$.
  This is the rectangle we are seeking to produce.
  See Figure~\ref{fig:Ladders} for a visualization.
  
  \begin{figure}[htp]
  \begin{center}
    \begin{tikzpicture}[scale=1.3]
      
      \clip (-1,-0.1) rectangle (6,9);
      
      \foreach \y [count=\cnt from 0] in {1,1.5,2.5,4,6,8.5,10.5,13}{
        \coordinate (p\cnt) at (1/5*\y,\y);
      }
    
      \fill [black!15] (-1,0) -- (0,0) -- (6,10) -- (-1,10);
      
      \foreach \p [count=\pp from 2] in {0,...,5}{
        \fill [black!25] (p\p) rectangle (p\pp);
      }

      \draw (-1,0) -- (3,0);
      \draw (0,0) -- (2,10);
      \draw (0,0) -- (6,10);
      
      \draw [dashed] (p0) -- (2,0);
      
      \foreach \p in {0,...,5}{
        \fill (p\p) circle (0.05cm);
      }
      
      \foreach \p [count=\pp from 2] in {0,...,5}{
        \draw [dotted] (p\p) rectangle (p\pp);
      }
      
      \node [left=0.4cm] at (p2) {$(i_k, t_k)$};
      \node [above right] at (2,0) {$(N, 0)$};
      \node [above left] at (0,0) {$(0, 0)$};
      
    \end{tikzpicture}
  \end{center}
  \caption{Geometric interpretation of the proof of Lemma~\ref{lem:AllGood}, not drawn to scale. Time increases upward. The light gray area is the outer region, and the dark gray rectangles denote the regions of $P_k$.}
  \label{fig:Ladders}
\end{figure}
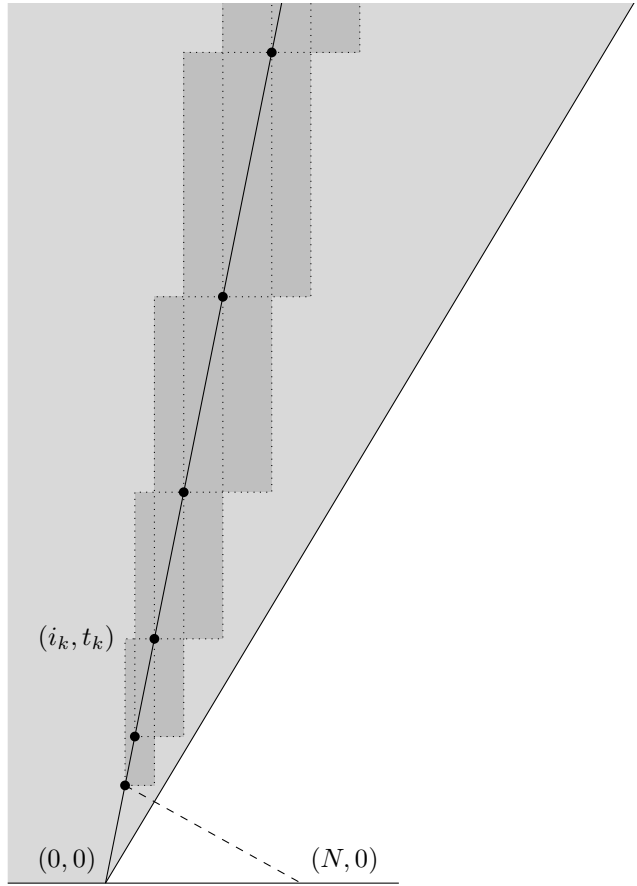
  
  We claim that if $P(k)$ holds for all $k \geq 0$, then $(i,t) \in A_\infty$ for all $t \in \N$ and $i \leq t (R_{\omega, a} + \theta_a)$.
  First, we have $(i,0) \in A_\infty$ for all $i \leq N$, and hence $(i,t) \in A_\infty$ for all $t \leq t_0 = \lceil \frac{N}{4 r} \rceil$ and $i \leq N - r t$ by the definition of the sets $A_k$.
  Since $R_{\omega,a} + \theta_a \leq 2 r$, this holds for all $t \leq t_0$ and $i \leq t (R_{\omega,a} + \theta_a)$ as long as $N$ is large enough.
  Hence the claim holds up to $t_0$.
  
  Suppose then that the claim holds up to some $t_k$.
  Since $P(k)$ holds, we have $(i,t) \in A_0$ for all $t_k \leq t \leq (1 + \rho \beta) t_k$ and $i_k \leq i \leq i_k + \rho \beta t_k (R_{\omega,a} + \theta_a)$.
  If $N$ (and hence $t_k$) is large enough, we have $(1 + \rho \beta) t_k > t_{k+1}$ and
  \[
    i_k + \rho \beta (R_{\omega,a} + \theta_a) t_k \geq t_k (\rho \beta + 1) (R_{a, \omega} + \theta_a) - 1 > i_{k+1} + r
  \]
  so that $(i,t) \in A_0$ for all $t_k \leq t \leq t_{k+1}$ and $i_k \leq i \leq i_{k+1} + r$.
  By the induction hypothesis, we also have $(i, t_k) \in A_\infty$ for all $i \leq t_k (R_{\omega, a} + \theta_a)$.
  For each $t_k \leq t < t_{k+1}$, the condition $(i, t) \in A_\infty$ for all $i \leq i_{k+1} + r$ implies $(i, t+1) \in A_\infty$ for all $i \leq i_{k+1}$ by the definition of the $A_k$, so by an inductive argument we obtain $(-\infty, i_{k+1}] \times [t_k, t_{k+1}] \subset A_\infty$.
  Since $i_{k+1} \geq t_{k+1} (R_{\omega,a} + \theta_a) \geq t (R_{\omega,a} + \theta_a)$ in this time interval, we have shown that the claim holds up to $t_{k+1}$.
  
  We now estimate the probability of the $P(k)$, given that $R(x)$ is $b$-good.
  We compute
  \begin{align*}
    j_k + C_k(\alpha + r \beta)
    \leq {}
    & t_k (R_{\omega,a} + \theta_a) + C_k (\alpha + r \beta - D_b) \\
    {} \stackrel{(*)}{=} {}
    & t_k (R_{\omega,a} + 2 \theta_a) - C_k (R_{\omega,a} + 2 \theta_a) \\
    {} \leq {}
    & s_k (R_{\omega,a} + 2 \theta_a)
  \end{align*}
  where the equality $(*)$ follows from $C_k = \rho t_k$.
  By the definition of $A_0$, this implies $(i_k + i, t_k + t) \in A_0$ for each $0 \leq i \leq (R_{\omega,a} + \theta_a) C_k \beta$ and $0 \leq t \leq C_k \beta$ -- which is exactly the event $P(k)$ -- with probability at least $1 - \lambda^{C_k^{Q_b/3}}$.
  From the union bound we obtain $\prob [ \forall k \geq 0 : P(k) \;|\; G_b] \geq 1 - \sum_{k = 0}^\infty \lambda^{C_k^{Q_b/3}}$, where $G_b$ is the event that $R(x)$ is $b$-good.
  
  We now compute $C_k^{Q_b/3} \geq (\frac{\rho N}{4 r} (1 + \rho \beta)^{k/2})^{Q_b/3} = \gamma N^{Q_b/3} \phi^k$, where $\gamma = (\frac{\rho}{4 r})^{Q_b/3}$ and $\phi = (1 + \rho \beta)^{Q_b/6}$ are constants independent of $N$ and $k$.
  For large enough $k$, we have $\phi^k > k$, so that
  \begin{align*}
  \sum_{k = 0}^\infty \lambda^{C_k^{Q_b/3}}
  \leq {} &
  \sum_{k = 0}^{k_0-1} \lambda^{\gamma N^{Q_b/3} \phi^k} + \sum_{k=k_0}^\infty (\lambda^{\gamma N^{Q_b/3}})^k \\
  {} = {} &
  \sum_{k = 0}^{k_0-1} \lambda^{\gamma N^{Q_b/3} \phi^k} + \frac{\lambda^{\gamma k_0 N^{Q_b/3}}}{1 - \lambda^{\gamma N^{Q_b/3}}}
  \end{align*}
  Since this is a finite sum with a constant number of terms, each of which is $\exp(-\Theta(N^{Q_b/3}))$, the claim follows.
\end{proof}

  

  

Together with the union bound, Lemma~\ref{lem:OuterRegion} and Lemma~\ref{lem:AllGood} imply that the probability of $R(x)$ being $\omega$-good and $b$-good is at least $1 - d^{N - K} - \xi^{N^{Q_b/3}}$ for large enough $N$.
We now choose $\xi < \delta_a < 1$ arbitrarily, and set $Q_a = Q_b/3$.
With these choices~\eqref{eq:InductiveStrong} holds, which finishes the proof of Proposition~\ref{prop:Inductive}.

\begin{proof}[Proof of second half of Theorem~\ref{thm:StableEroder}]
  Let $f : S^\Z \to S^\Z$ be a monotonic cellular automaton that does not satisfy the stability condition.
  Let $\epsilon > 0$ and let $R$ be the independent maximizing $\epsilon$-perturbation of $f$.
  By Lemma~\ref{lem:AltStability}, there exists a quiescent state $\omega \in S \setminus \{0\}$ such that $L_{a,\omega} \leq R_{\omega,a}$ for all $a < \omega$.
  By Proposition~\ref{prop:Inductive}, every quiescent state $a < \omega$ is inductive, so in particular $0$ is an inductive state.
  Thus there exist $0 < \delta_0 < 1$ and $Q_0 > 0$ such that, denoting $x = {}^\infty 0 \omega^N . \omega^N 0^\infty$, we have
  \[
    \prob [ \forall t \in \N : R(x)^t_{\lfloor t (L_{a,\omega} + R_{\omega,a}) / 2 \rfloor} = \omega ] \geq 1 - \delta_0^{N^{Q_0}}
  \]
  for all large enough $N$.
  Lemma~\ref{lem:ImmortalIsland} implies that $f$ is not a stable eroder.
\end{proof}

\section{Further Results}

We have presented a characterization of those one-dimensional monotonic cellular automata that erode finite islands in the presence of sufficiently low random noise.
The characterization was given in terms of forcing sets (Definition~\ref{def:StabCond} and Theorem~\ref{thm:StableEroderGeneral}), and alternatively in terms of Gal'perin rates (Theorem~\ref{thm:StableEroder}).
Since the Gal'perin rates of a one-dimensional monotonic CA can be computed from the local rule \cite{Ga77,dSaLeTo14}, we further obtain the following.

\begin{corollary}
  Given the local rule of a one-dimensional monotonic cellular automaton, it is decidable whether the automaton is a stable eroder.
\end{corollary}

In the context of probabilistic cellular automata, another interesting property is ergodicity.

\begin{definition}
  Let $S$ be a state set and $d \geq 1$, and let $R$ be a stochastic symbolic process on $S^{\Z^d}$.
  We say $R$ is \emph{ergodic}, if the marginals $R(\mu)^t$ converge weakly to the same measure on $S^{\Z^d}$ for every choice of $\mu \in \meas{S^{\Z^d}}$.
\end{definition}

Consider a one-dimensional monotonic CA $f$ on $S = \{0, \ldots, m\}$, and suppose that the states $0$ and $m$ are quiescent.
Let $R$ be an $\epsilon$-perturbation of $f$, and suppose that $R$ only introduces increasing errors (that is, $R(x)^{t+1} \geq f(R(x)^t)$ holds for all $t \in \N$).
Then we have $\lim_{t \to \infty} R(\uniform{m}) = \uniform{m}$.
If $f$ is also a stable eroder and $\epsilon$ is small enough, then $\lim_{t \to \infty} R(\uniform{0}) \neq \uniform{m}$, so $R$ is not ergodic.

In the converse direction, suppose that $f$ is not a stable eroder, so Lemma~\ref{lem:AltStability} gives us a quiescent state $\omega > 0$ with $L_{a,\omega} \leq R_{\omega,a}$ for all $a < \omega$.
We again choose $\omega$ to be minimal.
In this case, if $R$ is the independent $\omega, \epsilon$-perturbation of $f$, the results of Section~\ref{sec:Converse} imply that $\lim_{t \to \infty} \prob [\forall i \in C : R(\uniform{0})^t_i \geq \omega] = 1$ for any finite set $C \subset \Z$.
Since $f$ is monotonic, we can replace $\uniform{0}$ by an arbitrary measure and the result still holds.
If $\omega = m$, this implies that $R$ is ergodic.
However, if $\omega < m$, then $R$ is not ergodic, since $R(\uniform{0})^t$ converges to the point measure on $\uniform{\omega}$.
Furthermore, it can be the case that even the independent maximizing $m, \epsilon$-perturbation of $f$ is not ergodic for any $\epsilon$.
We show this by an example.

\begin{example}
  Recall the CA $f : S^\Z \to S^\Z$ from Example~\ref{ex:Ternary}.
  Let $T = \{0,1,2,3\}$, and for $a \in T$, denote $\bar{a} = \min(a, 2)$.
  Let $g : T^\Z \to T^\Z$ be the cellular automaton defined by the local rule
  \[
    G(a,b,c) =
    \begin{cases}
      F(\bar{a},\bar{b},\bar{c}), & \text{if~} b \leq 2, \\
      3, & \text{if~} a = b = c = 3, \\
      \bar{b}, & \text{otherwise.}
    \end{cases}
  \]
  The CA $g$ behaves exactly like $f$ on $S^\Z$.
  There is also a new state, $3$, that always erodes away at linear speed.
  It is easy to see that $g$ is a monotonic eroder, and it is not a stable eroder, since $f$ is not.
  Consider then the restriction of $g$ on $\{2,3\}^\Z$.
  This binary CA satisfies $L_{2,3} = 1 > -1 = R_{3,2}$, so it is a stable eroder.
  In particular, the independent $3, \epsilon$-perturbation $R$ of $g$ is not ergodic for any $\epsilon > 0$, since $\lim_{t \to \infty} R(\uniform{2})^t \neq R(\uniform{3})$.
  In this example, we have $\omega = 2$.
\end{example}

Finally, consider the two-dimensional cellular automaton $f$ with state set $\{0,1\}$ and neighborhood $N = \{(0,0), (1,0), (0,1)\}$, where the local rule always chooses the majority state in the three cells of $N$.
This automaton was first studied by Toom.
It is monotonic, and we can apply Proposition~\ref{prop:BinaryEroder} to show that it is a stable eroder.
In fact, since $f$ is symmetric with respect to switching the states $0$ and $1$, it is also a stable eroder toward $1$, meaning that if a small perturbation of $f$ is initialized on the all-$1$ configuration, the probability of any single cell to contain $0$ is low.
With our results, we can show that in the one-dimensional case, such an automaton does not exist.

\begin{proposition}
  Let $S = \{1, \ldots, m\}$, and let $f : S^\Z \to S^\Z$ be a monotonic CA.
  Let $g : S^\Z \to S^\Z$ be the CA defined by $g(x)_i = m + 1 - x_i$.
  If $f$ is a stable eroder, then $h = g \circ f \circ g$ is not an eroder.
\end{proposition}

\begin{proof}
  Since $f$ is an eroder, Theorem~\ref{thm:StableEroder} implies the existence of a quiescent state $a = a_{k-1} < m$ with $L_{a,m} > R_{m,a}$.
  If $h$ was an eroder, we would have $R_{m,b} > L_{b,m}$ for all $b < m$, which is impossible.
\end{proof}

On the other hand, there does exist a monotonic CA $f$ that is an eroder in both directions.

\begin{example}
  Let $S = \{0, 1, 2\}$, and let $f : S^\Z \to S^\Z$ be the radius-$2$ CA defined by the local rule
  \[
    F(a,b,c,d,e) =
    \begin{cases}
      0, & \text{if~} \max(a, b, c, d) \leq 1, e = 0, \\
      1, & \text{if~} c = 0, \min(d, e) \geq 1, \\
      1, & \text{if~} \max(d, e) \leq 1, c = 2, \\
      2, & \text{if~} a = 2, \min(b, c, d, e) \geq 1, \\
      c, & \text{otherwise.}
    \end{cases}
  \]
  A case analysis shows that $f$ is monotonic.
  By analyzing the behavior of $f$ on different steps, we can compute $R_{0,1} = -1 > -2 = L_{1,0}$ and $R_{0,2} = 1 > 0 = L_{2,0}$.
  By Theorem~\ref{thm:Galperin}, $f$ is an eroder.
  If we define $g(x)_i = 3 - x_{-i}$, then $g \circ f \circ g = f$, which means that the symbol-inverted version of $f$ behaves like the left-right-inverted version of $f$.
  In particular, the former is also an eroder.
  Of course, $f$ is not a stable eroder.
\end{example}

\bibliographystyle{plain}
\bibliography{bib}

\begin{thebibliography}{1}

\bibitem{ChLeRe79}
J.~Chalupa, P.~L. Leath, and G.~R. Reich.
\newblock Bootstrap percolation on a {Bethe} lattice.
\newblock {\em Journal of Physics C: Solid State Physics}, 12(1):L31, 1979.

\bibitem{dMeTo06}
Mois\'{e}s~Lima de~Menezes and Andr\'{e} Toom.
\newblock A non-linear eroder in presence of one-sided noise.
\newblock {\em Brazilian Journal of Probability and Statistics}, 20(1):1--12,
  2006.

\bibitem{dSaLeTo14}
Lu\'{i}s~Henrique de~Santana, Andr\'{e} Leite, and Andr\'{e} Toom.
\newblock Computing directional galperin's rates.
\newblock {\em Proceeding Series of the Brazilian Society of Applied and
  Computational Mathematics}, 2(1), 2014.

\bibitem{Ga76}
G.~A. Gal'perin.
\newblock One-dimensional automata networks with monotone local interaction.
\newblock {\em Problemy Pereda\v ci Informacii}, 12(4):74--87, 1976.
\newblock English translation: Problems of Information Transmission 12, no. 4,
  299--310 (1977).

\bibitem{Ga77}
G.~A. Gal'perin.
\newblock Rates of the propagation of the interaction in one-dimensional
  automata networks.
\newblock {\em Problemy Pereda\v ci Informacii}, 13(1):73--81, 1977.
\newblock English translation: Problems of Information Transmission 13, no. 1,
  52--58 (1977).

\bibitem{MaMa14}
Jean Mairesse and Ir\`{e}ne Marcovici.
\newblock Around probabilistic cellular automata.
\newblock {\em Theoretical Computer Science}, 559:42 -- 72, 2014.
\newblock Non-uniform Cellular Automata.

\bibitem{To76}
Andr\'{e} Toom.
\newblock Unstable multicomponent systems.
\newblock {\em Problemy Pereda\v ci Informacii}, 12(3):78--84, 1976.
\newblock English translation: Problems of Information Transmission 12, no. 3,
  220--225 (1977).

\bibitem{To80}
Andr\'{e} Toom.
\newblock Stable and attractive trajectories in multicomponent systems.
\newblock In {\em Multicomponent random systems}, volume~6 of {\em Adv. Probab.
  Related Topics}, pages 549--575. Dekker, New York, 1980.

\end{thebibliography}

\end{document}